\documentclass[preprint,12pt]{elsarticle}

\usepackage{amssymb}
\usepackage{amsmath}
\usepackage{amsthm}
\usepackage{array}
\usepackage{rotating}
\usepackage{vmargin}
\setmarginsrb{2.2cm}{1cm}{2.2cm}{3cm}{1cm}{1cm}{1.5cm}{1.5cm}
\newtheorem{thm}{Theorem}

\newtheorem{prop}{Proposition}
\newtheorem{lem}[thm]{Lemma}
\newtheorem{conj}{Conjecture}
\newdefinition{rem}{Remark}
\newdefinition{defi}{Definition}
\newproof{pf}{Proof}
\newcolumntype{M}[1]{>{\raggedright}m{#1}}

\def\Sp{\operatorname{Sp}}
\def\arctanh{\operatorname{arctanh}}
\def\dilog{\operatorname{dilog}}

\journal{*****}

\begin{document}

\begin{frontmatter}

\title{Integrable homogeneous potentials of degree $-1$ in the plane with small eigenvalues}
%\tnotetext[label1]{}
\author{Thierry COMBOT\fnref{label2}}
\ead{thierry.combot@u-bourgogne.fr}
%\ead[url]{home page}
%\fntext[label2]{}
%\cortext[cor1]{}
\address{IMB, Universi\'e de Bourgogne, 9 avenue Alain Savary, 21078 Dijon Cedex }
%\fntext[label3]{}

\author{}

\address{}

\begin{abstract}
We give a complete classification of meromorphically integrable homogeneous potentials $V$ of degree $-1$ which are real analytic on $\mathbb{R}^2\setminus \{0\}$. In the more general case when $V$ is only meromorphic on an open set of an algebraic variety, we give a classification of all integrable potentials having a Darboux point $c$ with $V'(c)=-c,\; c_1^2+c_2^2\neq 0$ and $\hbox{Sp}(\nabla^2 V(c)) \subset\{-1,0,2\}$. We eventually present a conjecture for the other eigenvalues and the degenerate Darboux point case $V'(c)=0$.
\end{abstract}

\begin{keyword}
%% keywords here, in the form: keyword \sep keyword
Morales-Ramis theory\sep Homogeneous potential \sep Central configurations \sep Differential Galois theory \sep Integrable systems

%% MSC codes here, in the form: \MSC code \sep code
%% or \MSC[2008] code \sep code (2000 is the default)

\end{keyword}

\end{frontmatter}

The problem of finding potentials which are integrable in the Liouville sense is a difficult and ancient problem. Liouville found that finding enough first integrals ($n$ for a $n$-dimensional potential) allows to solve the differential system associated to the potential by quadrature (the potential is then called integrable). The main difficulty is to find these first integrals, as they do not always exists, at least not globally. Almost all integrable rational potentials have rather simple first integrals, but one cannot even exclude very high degree rational first integrals. So one of the main problem is to find all integrable potentials, and certify that no others exist.

A Theorem from Morales-Ramis-Simo (Theorem 2 in \cite{19}) gives necessary conditions for integrability with meromorphic first integrals (see also earlier versions of this Theorem in \cite{1,4,5}). The differential Galois group of the variational equation near a non-trivial orbit should have an Abelian identity component. One difficulty is to find this non-trivial orbit, which led many authors to study homogeneous potentials. Indeed, aside physical interest, such potentials generically have straight line orbits, and then the Morales-Ramis-Simo Theorem can be applied. This procedure has been used in many non-integrability proofs \cite{2,3,6} and classifications. In particular, Maciejewski-Przybylska found all meromorphically integrable planar polynomial homogeneous potentials of degree $3,4$ in \cite{2,3}. In the case of the homogeneity degree $-1$, many results are linked to the $n$ body problem, which involves such homogeneous potentials of degree $-1$ \cite{8,9,24,25,26,27}.

In this article, we want to do a similar classification work for the homogeneity degree $-1$ as Maciejewski-Przybylska did for degree $3,4$ in \cite{2,3} in the plane. However, these articles generally consider polynomial potentials and we would like to extend the class of studied potentials to algebraic ones (e.g. to include n-body problems). In \cite{40}, we extended the Morales-Ramis-Simo theorem to the class of algebraic potentials; so we now recall our setting from \cite{40}.

Let $I=<P_1,\dots,P_s>$ a $2$-dimensional prime ideal of $\mathbb{C}[q_1,q_2,w_1,\dots,w_s]$ and $\Omega$ be a non-empty open set of $I^{-1}(0)$. Assume that the Jacobian of the application $w\mapsto (P_1(w),\dots,P_s(w))$ is maximal on $\Omega$, and that $\Omega$ has a ``homogeneity property''
$$\exists k_0,\dots,k_s \in\mathbb{Z},\; k_0\neq 0, \forall \alpha\in\mathbb{C}^*$$
$$(q,w)\in\Omega \Rightarrow (\alpha^{k_0} q_1,\alpha^{k_0} q_2,\alpha^{k_1} w_1,\dots, \alpha^{k_s} w_s)\in\Omega$$
On this $\Omega$, we can define a holomorphic homogeneous function $V$, which will be our potential. Remark that this class of potentials includes $V(q_1,q_2)=(q_1^2+q_2^2)^{-1/2}$ (which is integrable), by taking
$$\Omega=\{(q_1,q_2,w_1)\in\mathbb{C}^2\times \mathbb{C}^*,\;\;w_1^2-q_1^2-q_2^2=0\},\;\; V=w_1^{-1}$$
We will not succeed in finding all such meromorphically integrable homogeneous potentials, and thus we will add a technical assumption on eigenvalues at Darboux points, which, as we will see in Section \ref{secNon}, is of crucial importance in such integrability analysis.\\

The main theorems of this article are the following

\begin{thm}\label{thmmain1}
Let $V$ be a real analytic potential on $\mathbb{R}^2\setminus \{0\}$, homogeneous of degree $-1$. If $V$ is meromorphically integrable, then
$$V=\frac{a}{r}\quad a\in\mathbb{R}$$
\end{thm}

\begin{thm}\label{thmmain2}
Let $V$ be a holomorphic homogeneous potential of degree $-1$ on $\Omega\subset \mathcal{S}$. We assume that there exists $c\in\Omega$ such that $V'(c)=-c$, $c_1^2+c_2^2\neq 0$ and the spectrum of the Hessian matrix of $V$ at $c$ satisfies $\Sp(\nabla^2(V)(c))\subset\{-1,0,2\}$. If $V$ is meromorphically integrable, then $V$ belongs after rotation to one of the following families
$$V=\frac{a}{q_1}+\frac{b}{q_2},\quad V=\frac{a}{r},\quad V=\frac{aq_1}{(q_1+\epsilon iq_2)^2},\quad a,b\in\mathbb{C},\; \epsilon=\pm 1$$
$$\hbox{ with } \mathcal{S}=\{(q_1,q_2,r)\in\mathbb{C}^3,\;\;r^2=q_1^2+q_2^2\}$$
\end{thm}

This last unexpected case was found by Hietarinta in \cite{7}. However, not all homogeneous potentials satisfy these hypotheses. In particular, the condition $\Sp(\nabla^2(V)(c))\subset\{-1,0,2\}$ is very restrictive. However, we conjecture thank to computer computations that there are no integrable potentials in the plane with other eigenvalues.

\bigskip
In section \ref{Int}, we present variational equations and the Morales-Ramis-Simo Theorem.\\
In section \ref{secGen}, we present several properties of homogeneous potentials, in particular the notion of meromorphic integrability for homogeneous potentials on $\mathcal{S}$ and polar coordinates.\\
In section \ref{secTheo}, we prove that Theorem \ref{thmmain2} easily implies Theorem \ref{thmmain1}.\\
Section \ref{secNon} presents some properties of higher variational equations, and in particular a notion of non-degeneracy. We prove that this property is often satisfied by higher variational equations, and implies a uniqueness Theorem \ref{thm4}.\\
Section \ref{secCas} deals with the special case of eigenvalue $-1$ for which this non-degeneracy property is not satisfied, and thus requiring a clever analysis of higher variational equations.\\

\section{Introduction to Morales-Ramis-Simo theorem}\label{Int}

The main idea of the Morales-Ramis-Simo theorem below is that if a Hamiltonian system is meromorphically integrable, then the linearised system along a particular solution should also be ``integrable''. In this section, the Hamiltonian $H$ will only be assumed to be a $n$ degrees of freedom Hamiltonian over a general $2n$ dimensional complex analytic manifold $M$.

Let us consider a holomorphic function $f$ on $T^*M$, and a point $x\in T^*M$. The initial form of $f$ at $x$ is the lowest order non-zero term in the Taylor expansion of $f$ at $x$. It is in particular a homogeneous polynomial. If $f$ is a meromorphic function on $T^*M$, then its initial form is defined as the quotient of the initial form of its numerator and denominator.

This definition can then be generalized to curves. Given a complex analytic curve $\Gamma \subset T^*M$ parametrized by $t$, we consider for a holomorphic $f$ the Taylor expansion of $f$ at $x(t)$. The coefficients of this expansion are functions of $t$, and the initial form is the lowest order non-zero term (as a function of $t$) in this expansion. Remark that the valuation of $f$ at $x(t)$ can differ for some exceptional values of $t$ (this typically occurs at singular points of the variational equation). In the general meromorphic case, the initial form of $f$ on $\Gamma$ is then a homogeneous rational fraction with coefficients depending on $t$.

\begin{lem}[Ziglin, (look Audin \cite{44})]
Let $f_1,\dots,f_k$ be germs of functionally independent meromorphic functions over a neighbourhood of $0$ in $\mathbb{C}^n$. Then there exist polynomials $P_1,\dots,P_k \in\mathbb{C}[z_1,\dots,z_n]$ such that the initial forms at the origin of the functions $g_i=P_i(f_1,\dots,f_k)$ are rational fractions algebraically independent in $\mathbb{C}(z_1,\dots,z_k)$.
\end{lem}

Let us consider $\Gamma \subset T^*M$ a trajectory of the Hamiltonian field $X_H$. If this field has $n$ independent first integrals $f_1,\dots,f_n$, then after possibly algebraic transformations, the initial forms of these first integrals can be assumed to be independent thanks to Ziglin Lemma. Remark that if the Poisson bracket $\{f_i,f_j\}=0$, then so it is for their initial forms.

We now define variational equations, following Morales-Ramis-Simo \cite{19} page 860. Let us note $\varphi_t$ the flow of the Hamiltonian field $X_H$. We note
$$\varphi_t(y)=\sum\limits_k \varphi_t^{(k)}(x) (y-x)^k$$
the series expansion of $\varphi_y$ at $x$. We define accordingly
$$X_H(y)=\sum\limits_k X_H^{(k)}(x) (y-x)^k$$
The variational equations can now be written in a compact form
$$\dot{\varphi}_t^{(1)}=X_H^{(1)}\varphi_t^{(1)}$$
$$\dot{\varphi}_t^{(2)}=X_H^{(1)}\varphi_t^{(2)}+X_H^{(2)}(\varphi_t^{(1)})^2$$
$$\dot{\varphi}_t^{(3)}=X_H^{(1)}\varphi_t^{(3)}+2X_H^{(2)}(\varphi_t^{(1)},\varphi_t^{(2)})+X_H^{(3)}(\varphi_t^{(1)})^3$$
and the general formula is given by
$$\dot{\varphi}_t^{(k)}= \sum\limits_{j=1}^k \sum \frac{j!}{m_1! \dots m_s!} X_H^{(j)}((\varphi_t^{(i_1)})^{m_1},\dots,(\varphi_t^{(i_s)})^{m_s})$$
The point derivation correspond to the derivation with respect to time along a particular solution $\Gamma$ of $X_H$.\\

If the Hamiltonian system admits a first integral $f$, then the first order variational equation admits a rational first integral, the initial form of $f$. The same holds for higher variational equations: Noting $f_k$ the series expansion of $f$ consisting of the $k$ first terms of the Taylor series of $f$ beginning by the first non-zero term, we obtain for $f_k$ a polynomial (or rational fraction in the case $f$ meromorphic) which is a first integral of the $k$-th order variational equation.\\

Remark that the first order variational equation is a linear one, but higher order ones are not. These however can be ``linearized'' by the following process:
\begin{itemize}
\item The right term of each equation in $\varphi_t^{(k)}$ is a polynomial in $\varphi_t^{(i)}$ with $i<k$. So the rightside of the equation is  a linear combination of monomial in $\varphi_t^{(i)}$ with $i<k$.
\item A product, power of a solution of a linear differential equation is itself solution of a linear differential equation (called symmetric power/product).
\item We can thus replace each monomial of the righthandside by a new unknown function, will be a solution of a linear differential equation
\end{itemize}
The $k$-th order variational equation (non linear version) can then be replaced by a linear differential system, whose solutions are the same than the linear version.\\

In the following, when considering higher variational equation (i.e. $k>1$), we will always consider the linearised version. In section \ref{secHig}, we present a simple way to build in practice this linearised $k$-th order variational equation in our particular case (a potential with $2$ degrees of freedom).\\

To the $k$-th variational equation, we can associate a Galois group. This Galois group preserves all rational invariants of the differential system, and in particular the rational invariants coming form first integrals of the Hamiltonian field. Through this process comes the following constraint on these Galois groups.

\begin{thm}\label{thmmorales0} (Morales-Ramis-Simo \cite{19})
Let $H$ be a Hamiltonian over a complex analytic symplectic manifold $M$ of dimension $2n$. Assume $H$ is meromorphically integrable in the Liouville sense ($X_H$ admits $n$ independent meromorphic first integrals, pairwise Poisson commuting). Let $\Gamma$ be a connected not reduced to a point particular solution of $X_H$. Then the identity component of the Galois group of the $k$-th order variational equation near $\Gamma$ is Abelian over the base field of meromorphic functions on $\Gamma$.
\end{thm}

\section{Homogeneous potentials on algebraic manifolds}\label{secGen}
\subsection{Definitions}

We will consider from now a Hamiltonian system given by
$$H(p_1,p_2,q_1,q_2,\mathbf{w})=\frac{1}{2}(p_1^2+p_2^2) -V(q_1,q_2,\mathbf{w}) \;\;\hbox{ with }\;\; (q_1,q_2,\mathbf{w})\in\Omega\subset\mathcal{S}$$
The Hamiltonian $H$ is associated to a dynamical system $X_H$ on $\mathbb{C}^2\times \Omega$. The open set $\Omega$ is a subset of an algebraic variety $\mathcal{S}$, which corresponds to the space of positions. It projects ``well'' on $\mathbb{C}^n$ in the sense that the symplectic structure on $\mathbb{C}^2\times\Omega$ defined by the derivations in $p,q$ does not degenerate (look at page 2 or \cite{40} for more precisions). This Hamiltonian is a $2$ degrees of freedom system, and $H$ is holomorphic on $\mathbb{C}^2\times \Omega$.
\begin{defi}\label{defhom}
We say that a holomorphic potential on $\Omega\subset \mathcal{S}$ is homogeneous of degree $-1$ if for all $(q_1,q_2,\mathbf{w})\in\Omega,\; \alpha\in\mathbb{C}^*$
$$V(\alpha^{k_0} q_1,\alpha^{k_0} q_2,\alpha^{k_1} w_1,\dots, \alpha^{k_s} w_s) = \alpha^{-k_0} V(q_1,q_2,\mathbf{w})$$
\end{defi}

\bigskip

In the rest of the article, the potential $V$ will now be assumed to be holomorphic on $\Omega$ and homogeneous of degree $-1$. This type of potentials contains many useful algebraic potentials such as potentials in celestial mechanics (which often contains square roots due to the mutual distances appearing in the potential). The construction of this Hamiltonian system is the one we introduced in \cite{40} to define algebraic potentials.
Remark that according to the definition \ref{defhom}, it is always possible to multiply all the $k_i$ by an integer. Still, we cannot normalize the $k_0$ always to $1$. Indeed, to allow $k_0>1$ is necessary if we want to include algebraic extension of rational homogeneity degree, as $w_1^4=q_1^2+q_2^2$ (here we obtain $k_0=2,k_1=1$). A motivation to consider homogeneous potentials on (open sets of) algebraic manifolds instead of $\mathbb{C}^2$ is that we want to include the rotation-invariant potential $V=1/r$, which is always integrable in the sense of the following definition

\begin{defi}\label{def1}
Let $V$ be a holomorphic homogeneous potential of degree $-1$ on $\Omega$. We will say that $V$ is meromorphically integrable if there exists a first integral $I$ of $X_H$ meromorphic on $\mathbb{C}^2\times \Omega$ and functionally independent with $H$.
\end{defi}

In our particular setting (a potential with $2$ degrees of freedom over $\Omega \subset \mathcal{S}$), the Morales-Ramis-Simo theorem can be rewritten on the following form
\begin{thm}\label{thmmorales} (Combot \cite{40} Theorem 2. page 3)
Let $V$ be a holomorphic potential on an open set $\Omega\subset \mathcal{S}$ and $\Gamma\subset \mathbb{C}^2\times \Omega$ a non-stationary orbit of $V$. Assume that $\Omega\cap \Sigma(\mathcal{S})=\emptyset$.  If there are two first integrals meromorphic on $\mathbb{C}^2\times \Omega$ of $V$ that are in involution and functionally independent over an open neighbourhood of $\Gamma$, then the identity component of Galois group of the variational equation near $\Gamma$ is Abelian over the base field of meromorphic functions on $\Gamma$.
\end{thm}

Remark that in the original statement of \cite{40}, the particular curve $\Gamma$ is assumed to be not included in the singular set $\Sigma(\mathcal{S})$, but here we have already removed this singular set out of $\Omega$.

\subsection{Darboux points}

In Theorem \ref{thmmorales}, a key ingredient is the orbit $\Gamma$. To find such an orbit of our Hamiltonian system, we will use Darboux points.

\begin{defi}\label{def2}
Let $V$ be a holomorphic homogeneous potential of degree $-1$ on $\Omega$. We say that $c\in\Omega\setminus \{0\}$ is a \emph{Darboux point} of $V$ if
\begin{equation}\label{darbb}
\frac{\partial}{\partial q_1} V(c)= \alpha c_1\qquad \frac{\partial}{\partial q_2} V(c)= \alpha c_2
\end{equation}
The number $\alpha\in\mathbb{C}$ is called the \emph{multiplier} associated to $c$. We say that $c$ is non-degenerate (or proper in \cite{3}) if $\alpha\neq 0$.
\end{defi}

In the non-integrability setting, these Darboux points are also used in \cite{1,2,3,8,21} among others. Using homogeneity of $V$, we can always choose $\alpha\in\{0,-1\}$ and so in the following we will always choose the multiplier $\alpha=-1$ for a non-degenerate Darboux point (in which case we say hat the Darboux point is normalized). The most interesting property for us of these Darboux points is that they provide orbits:

\begin{defi}\label{def22}
Let $V$ be a holomorphic homogeneous potential of degree $-1$ on $\Omega$. Let $c\in\Omega$ be a Darboux point of $V$. A \emph{homothetic orbit of $V$ associated to $c$} is given by
$$q_i(t)=c_i\phi(t)^{k_0}\quad p_i(t)=c_i k_0\dot{\phi}(t) \phi(t)^{k_0-1}\quad i=1, 2$$
$$w_i(t)=c_{i+2} \phi(t)^{k_i}\quad i=1\dots s$$
with $\phi$ satisfying the following differential equation
$$\frac{1} {2} (k_0\dot{\phi} \phi^{k_0-1})^2=-\frac{\alpha} {\phi^{k_0}}+E \quad E\in\mathbb{C}$$
\end{defi}

This homothetic orbit is used by Morales-Ramis in \cite{21} to build simple integrability conditions thanks to the classification of Galois groups of the hypergeometric equation by Kimura \cite{14}. Along a homothetic orbit, the first order variational equation is given by
$$\ddot{X}=\frac{1}{\phi(t)^{3k_0}} \nabla^2 V(c) X$$
where $\nabla^2 V(c)$ is the Hessian matrix of $V$ at $c$. As the potential is homogeneous, multiplying the value of $E$ does not change the variational equation (up to a change of variable), and so we can always choose $E\in\{0,1\}$. The case $E=0$ does not lead to any integrability constraint, and so we will only consider $E=1$ in the rest of the article.

\bigskip

After the variable change $k_0\dot{\phi(t)} \phi(t)^{k_0-1}/\sqrt{2} \longrightarrow t$ and diagonalization of $\nabla^2 V(c)$ (when possible), the first order variational equation becomes
$$\frac{1}{2}(t^2-1) \ddot{X}_i +2t\dot{X}_i -\lambda_i X_i=0, \qquad\quad \lambda_i\in \Sp\left(\nabla^2 V(c)\right)$$
The integrability condition is that the Galois group of this variational equation over the base field of meromorphic functions on the curve $\Gamma$ should have an Abelian identity component. Here the base field after reparametrization is $\mathbb{C}(t,\sqrt{1+t^{-1}})$. This integrability condition on the Galois group leads to a condition on the eigenvalues $\lambda_i$, which is (according to Morales-Ramis in \cite{21} and Combot \cite{40})
$$\Sp\left(\nabla^2 V(c)\right)\subset \left\lbrace\frac{1}{2}(k-1)(k+2),\; k\in\mathbb{N} \right\rbrace=\{-1,0,2,5,9,14,20,27,\dots \}$$

\begin{defi}\label{def23}
Let $V$ be a holomorphic homogeneous potential of degree $-1$ on $\Omega$. Let $c\in\Omega$ be a Darboux point of $V$. We say that \emph{$V$ is integrable at order $k$ at $c$} if the variational equation of order $k$ of the homothetic orbit associated to $c$ has a Galois group whose identity component is Abelian.
\end{defi}

\subsection{Polar coordinates}

Let us first remark that we can always assume that $P_1(q,w)=q_1^2+q_2^2-w_1^2$. Indeed, if $w_1$ is not already a rational fraction on $\mathcal{S}$, we can always consider an algebraic variety $\tilde{\mathcal{S}}$ with a projection on $\mathcal{S}$
$$ \pi:(q_1,q_2,w_1,\dots,w_s) \mapsto (q_1,q_2,w_2,\dots,w_s)\quad \pi(\tilde{\mathcal{S}})=\mathcal{S}$$
In particular, the rational functions on $\tilde{\mathcal{S}}$ are rational functions on $\mathcal{S}$ and rational in $w_1$, the additional algebraic extension. This will of course define a potential $\tilde{V}$ on $\tilde{\mathcal{S}}$ such that
$$\tilde{V}(q_1,q_2,w_1,\dots,w_s)= V(q_1,q_2,w_2,\dots,w_s)$$
So the only consequence of this construction will be that
$$\pi(\Sigma(\tilde{\mathcal{S}}))=\{q_1=\pm iq_2,\; (q,w)\in\mathcal{S}\} \cup \Sigma(\mathcal{S})$$
Still, this will not have consequence in the proofs of Theorems \ref{thmmain1},\ref{thmmain2}, due to the hypothesis $c_1^2+c_2^2\neq 0$.

\bigskip

So in the rest of the article, we will always assume that $P_1(q,w)=q_1^2+q_2^2-w_1^2$. The algebraic extension $w_1$ will be noted for now $r$ (as usual). Let us now define polar coordinates. Recall that the manifold $\mathcal{S}$ has the following homogeneity property
$$\forall \alpha\in\mathbb{C}^*, (q,r,\mathbf{w})\in\mathcal{S} \Rightarrow (\alpha^{k_0}q,\alpha^{k_0}r,\alpha^{k_1}\mathbf{w}_1,\dots,\alpha^{k_s}\mathbf{w}_s)\in\mathcal{S}$$
The $k_i$ are integers related to the homogeneity of the algebraic extensions $\mathbf{w}_i$. Noting
$$r\cos\theta=q_1,\;\;r\sin\theta=q_2$$
and using the homogeneity of $V$, we obtain that $V$ can be written
$$V(q_1,q_2,r,\mathbf{w})=r^{-1} \tilde{U}(\cos \theta,\sin \theta,\tilde{\mathbf{w}})$$
where $\tilde{\mathbf{w}}_i^{k_0}=\mathbf{w}_i/r^{k_i}$ and $\tilde{U}$ is holomorphic for $(\cos \theta,\sin \theta,1,\tilde{\mathbf{w}})\in\Omega$. In polar coordinates, a rotation of a potential $V=r^{-1}\tilde{U}(\cos \theta,\sin \theta,\tilde{\mathbf{w}})$ of an angle $\theta_0\in\mathbb{C}$ is simply the potential $V=r^{-1}\tilde{U}(\cos (\theta+\theta_0),\sin (\theta+\theta_0),\tilde{\mathbf{w}})$ (applying also the rotation to the ideal defining the $\tilde{\mathbf{w}}$).

\bigskip

Let us now consider a point $c\in\Omega$. As $\Omega \cap \Sigma(\mathcal{S})= \emptyset$, the potential $V$ is holomorphic on a neighbourhood of $c$, and thus also on an open neighbourhood $W\subset \Omega$ of
$$\tilde{c}=(c_1/c_3,c_2/c_3,1,c_4/c_3^{k_1/k_0},\dots,c_{s+3}/c_3^{k_s/k_0})$$
(whatever the possible choice of the root $c_3^{k_i/k_0}$). We can write
$$\tilde{U}(\cos \theta,\sin \theta,\tilde{\mathbf{w}})=V(\cos \theta,\sin \theta,1,\mathbf{w})$$
This $\tilde{U}$ is not a function of only $\theta$ as it would be multivalued due to the algebraic extensions $\tilde{\mathbf{w}}$. But we can choose the branch on which $\tilde{c}$ lies. Let us first define the projection
$$\pi: W \mapsto \mathbb{C},\quad \pi(\cos \theta,\sin \theta,\mathbf{w})= \theta$$
The function $\pi$ is injective if the open neighbourhood $W$ is chosen small enough. We can then define the holomorphic function
$$U(\theta)=\tilde{U}(\pi^{-1}(\theta))$$
So, on $W$, we can always write
$$V(q_1,q_2,r,\mathbf{w})=r^{-1} U(\theta),\;\;\; (\cos \theta,\sin \theta,\mathbf{w}) \in W$$
with $U$ holomorphic in $\theta$.

\bigskip

Let us now look in the case of a Darboux point. A Darboux point $c$ of a holomorphic homogeneous potential $V$ of degree $-1$ on $\Omega$ is solution of equation \eqref{darbb}. If $c_1^2+c_2^2\neq 0$ (which is a necessary assumption for dealing with polar coordinates), we can rewrite this equation in polar coordinates with $V=r^{-1} U(\theta)$ on an open neighbourhood of $c$
$$U'(\theta)=0 \quad \alpha r^3=-U(\theta)$$
where $'$ denotes the derivation in $\theta$. So a non-degenerate Darboux point corresponds to some $\theta\in[0,2\pi [$ such that
$$U'(\theta)=0 \quad \hbox{and} \quad U(\theta)\neq 0$$

\subsection{Reduction by rotation}

Given a $2$-dimensional rotation $R_{\theta_0}$ of angle $\theta_0$, the symplectic variable change $p=R_{\theta_0} p,\; q=R_{\theta_0} q$ transforms $H$ into the Hamiltonian of the meromorphic homogeneous potential $V(R_{\theta_0} q)$. So meromorphic integrability of the potential $V(R_{\theta_0} q)$ does not depend on $\theta_0$. Looking at the Hamiltonian flow $X_H$, by making a time change we can replace the potential $V$ by $\gamma V$ with $\gamma\in \mathbb{C}^*$ (we will call this transformation a dilatation), like Maciejewski-Przyzbylska in \cite{3}. So meromorphic integrability of the potential $\gamma V$ does not depend on $\gamma\in\mathbb{C}^*$.\\

Assume that $V$ has a non-degenerate Darboux point $c$; we can assume that $c$ is of the form $c=(1,0,\dots)$ and has multiplier $-1$, which corresponds in polar coordinates to $U'(\theta)=0,\; U(\theta)=1$.

\begin{lem}
Let $V$ be a holomorphic homogeneous potential of degree $-1$ on $\Omega$. Assume that $V$ admits a non-degenerate Darboux point $c\in\Omega$. Then after a rotation and dilatation, we can assume that the potential $V$ has the following properties
\begin{itemize}
\item There exists a point of the form $c=(1,0,\dots)$ which is a non-degenerate Darboux point of $V$ with multiplier $-1$.
\item We have $\Sp(\nabla^2 V(c))=\{2,\lambda\}$, and the series expansion in $q$ of $V$ at $c$ is of the form
$$V(c+q)=1-q_1+q_1^2+\lambda q_2^2/2 +O(q^3)$$
\end{itemize}
\end{lem}

\begin{proof}
As there exists a non-degenerate Darboux point $c\in\Omega$, we can assume that $c$ is of the form $c=(1,0,\dots)$ after a rotation (recall that $c_3^2=c_1^2+c_2^2\neq 0$ on $\Omega$). Multiplying $V$ by a constant, we can assume that $V(c)=1$ (recall that $V(c)\neq 0$ as $c$ is non degenerate). Using Euler formula, we obtain $\partial_{q_1}V(c)=-V(c)$ and so the multiplier of $c$ is $-1$.

Differentiating the Euler relation and evaluating it at $(q_1,q_2)=(c_1,c_2)$, we also have
$$\partial_{q_1}V(c)+\partial_{q_1q_1}V(c)=-\partial_{q_1}V(c),\qquad \partial_{q_1q_2}V(c)+\partial_{q_2}V(c)=-\partial_{q_2}V(c) $$
Thus
$$\nabla^2 V(c) (c_1,c_2)= \left(\begin{array}{c}\partial_{q_1q_1}V(c)\\ \partial_{q_1q_2}V(c)  \end{array} \right)= \left(\begin{array}{c}-2\partial_{q_1}V(c)\\ -2\partial_{q_2}V(c)  \end{array} \right)=2\left(\begin{array}{c} c_1\\ c_2  \end{array} \right)$$
So the eigenvalue $2$ always appear in the spectrum and $\Sp(\nabla^2 V(c))=\{2,\lambda\}$. The series expansion of $V$ at $c$ follows.
\end{proof}

\subsection{Example}

$$U(\theta)=(1-\cos(\theta))^n-\frac{n2^n}{(2k-1)(k+1)+1}-2^n\qquad n,k\in\mathbb{N}^*$$
The Darboux points of the potential $V=r^{-1} U(\theta)$ correspond to $\theta=0,\pi$. Computing the eigenvalues at these Darboux points gives respectively the following spectrum of Hessian matrices
$$\{2,-1\} \quad \{2,(2k-1)(k+1)\}$$
These eigenvalues are allowed for meromorphic integrability using Theorem \ref{thmmorales} and according to \cite{18}, there are no additional integrability conditions at order $2$. So this potential is integrable at order $2$ near all Darboux points. Moreover, looking at $\theta=0$, we find that
$$U^{(i)}(0)=0 \quad i=1\dots 2n-1$$
This implies that the variational equation of order $2n-2$ of the potential $V=r^{-1} U(\theta)$ is the same as the variational equation of order $2n-2$ of the potential $\tilde{V}=r^{-1}$. This potential $\tilde{V}$ is meromorphically integrable with an additional first integral $p_1q_2-p_2q_1$ and thus its variational equation of order $2n-2$ has a Galois group whose identity component is Abelian. So the potential $V$ is integrable at order $2n-2$ at $\theta=0$.
At $\theta=\pi$, the potential $V$ is probably not integrable at order $3$ but it seems quite difficult to prove as the eigenvalue depend on the parameter $k$ which make the higher variational equation very complicated (this problem is analysed in \cite{39}).

We could also use the procedure presented in \cite{2,3} where Maciejewski-Przybylska classify meromorphically homogeneous potentials of degree $3,4$, but in the case of $V$ this will not work because this method only works for potentials without multiple Darboux points (here the Darboux point corresponding to $\theta=0$ is multiple for $n\geq 2$). In section $5$, we will prove that the potential $V$ is not integrable at order $4n-3$ at $\theta=0$.

\section{Theorem \ref{thmmain2}  implies Theorem \ref{thmmain1}}\label{secTheo}

\begin{lem}\label{thm3}
Let $V$ be a real analytic potential on $\mathbb{R}^2\setminus \{0\}$, homogeneous of degre $-1$. Then $V$ can be written in polar coordinates under the form $r^{-1}U(\theta)$ with $U$ $2\pi$-periodic real analytic, and there exists $\theta_0$ such that
$$U(\theta_0)\neq 0 \quad U'(\theta_0)=0 \quad \frac{U''(\theta_0)}{U(\theta_0)} \leq 0$$
\end{lem}

\begin{proof}
As $V$ is real analytic on $\mathbb{R}^2\setminus \{0\}$, it is also real analytic on the unit circle. Using homogeneity, we then have $V(q_1,q_2)=r^{-1} U(\theta)$. The function $U$ is thus real analytic $2\pi$-periodic. We have that $U(\mathbb{R})\subset \mathbb{R}$ and so $U$ is $C^\infty$ on $\mathbb{R}$. The function $U$ is periodic, so there exists a minimum and a maximum for $U$. Assume first that $U$ is not constant. Then $\max U > \min U$. We have $3$ cases
\begin{itemize}
\item $\max U \geq \min U \geq 0$. Then we choose $\theta_0$ such that $U(\theta_0)= \max U$
\item $\max U \geq 0 \geq \min U$. If $\max U\neq 0$, we choose $\theta_0$ such that $U(\theta_0)= \max U$, otherwise we choose $\theta_0$ such that $U(\theta_0)= \min U$
\item $0 \geq \max U \geq \min U$. We choose then $\theta_0$ such that $U(\theta_0)= \min U$
\end{itemize}
Knowing that $\max U > \min U$, we get $U(\theta_0) \neq 0$. Then in all cases, we have
$$\frac{U''(\theta_0)}{U(\theta_0)} \leq 0$$
Knowing that $\theta_0$ is an extremum, we get
$$U(\theta_0)\neq 0 \quad U'(\theta_0)=0 \quad \frac{U''(\theta_0)}{U(\theta_0)} \leq 0$$
which gives the Lemma.
\end{proof}

Let us now prove Theorem~\ref{thmmain1}, assuming Theorem~\ref{thmmain2}.
\begin{proof}[Proof of Theorem \ref{thmmain1}]
We assume that Theorem \ref{thmmain2} holds. As $V$ is real analytic on $\mathbb{R}^2\setminus \{0\}$, then $V$ is holomorphic over a neighbourhood of $\mathbb{R}^2\setminus \{0\}$ in $\mathbb{C}^2$, noted $W$. As $V$ is homogeneous, we can assume that $W$ is invariant by dilatation $q \mapsto \alpha q$.

Let us note $\Omega$ the open set of $\mathcal{S}=\{(q_1,q_2,r)\in\mathbb{C}^3,\;\;r^2=q_1^2+q_2^2\}$ such that
$$\Omega=\{(q,r),q\in W, (q,r)\in\mathcal{S},q_1^2+q_2^2\neq 0\}$$
The set $\Omega$ satisfies the conditions of Theorem \ref{thmmain2}, and $V$ is a holomorphic potential on $\Omega$. We can moreover write $V=r^{-1}U(\theta)$ in polar coordinates with $U$ real analytic and we use Lemma~\ref{thm3}. There exists a $\theta_0\in\mathbb{R}$ such that
$$U(\theta_0)\neq 0 \quad U'(\theta_0)=0 \quad \frac{U''(\theta_0)}{U(\theta_0)} \leq 0$$
We consider the point in $\Omega$
$$c_1 = U(\theta_0) \cos\theta_0, \quad  c_2 = U(\theta_0) \sin\theta_0, \quad r=U(\theta_0)$$
After computation, we find that $c$ satisfies the equation
$$\partial_{q_1} V(c)=-c_1\qquad \partial_{q_2} V(c)=-c_2$$
So $c$ is a Darboux point of $V$ with multiplier $-1$. We now write
$$V=r^k U\left(\hbox{arctan}\left(\frac{q_1+iq_2}{r}\right)\right)$$
We compute the Hessian matrix, evaluate at $c$, and we find for the spectrum
$$\Sp(\nabla^2V(c))=\left\lbrace 2,\frac{U''(\theta_0)}{U(\theta_0)}-1 \right\rbrace $$
If $V$ is meromorphically integrable then, thanks to Theorem~\ref{thmmorales}, the eigenvalues at Darboux points should belong to
$$\left\lbrace\frac{(p-1)(p+2)}{2},\;p\in\mathbb{N} \right\rbrace =\{-1,0,2,5,9,14,\dots \}$$
But here we have moreover that
$$\frac{U''(\theta_0)}{U(\theta_0)} \leq 0$$
So this implies that in fact
$$\Sp(\nabla^2 V(c))=\{2,-1\}$$

To conclude, the holomorphic potential on $\Omega$ satisfies all hypotheses of Theorem \ref{thmmain2}, including the spectrum condition. Among the $3$ possibles families, only the second one has the eigenvalue $-1$. Using the fact that $U$ should be a real function (and non zero), this implies that $V=a r^{-1},\; a\in\mathbb{R}^*$.
\end{proof}

Remark that the meaning of meromorphic integrability depends on the open set $\Omega$ chosen. We can choose an dilatation invariant open set $\Omega$ arbitrary small. So we have proved the non-existence of an additional first integral (outside the case $V=ar^{-1}$) which is meromorphic on $\mathbb{C}^2\times \Omega$ with $\Omega$ an arbitrary small dilatation-invariant neighbourhood of $\mathbb{R}^2\setminus \{0\}$.

%\bigskip
%
%Lemma \ref{thm3} has then proved the following
%$$\Lambda\left( \left\lbrace V \hbox{ real analytic on } \mathbb{R}^2\setminus \{0\},\;\; V\neq 0 \right\rbrace \right) = -1 $$
%So this set satisfies our ``bounded eigenvalue'' property. As we see, it is not necessary to have a finite dimensional family of homogeneous potentials to have the ``bounded eigenvalue'' property, and so large families could be analyzed.
%

\section{Non degeneracy of higher variational equations}\label{secNon}

The purpose of this section is to prove Theorem \ref{thmmain2} for the eigenvalues $0,2$. The strategy is the following
\begin{itemize}
\item We consider a potential $V$ which is integrable at a Darboux point $c=(1,0,\dots)$ at order $k-1$.
\item We rewrite the $k$-th order variational equation, and we prove that the constraints of being integrable at order $k$ are affine equations in the $k+1$-th derivative in $q_2$ of $V$.
\item If these affine functions are not constant, then there is at most one possibility for the $k+1$-th derivative in $q_2$ of $V$ for $V$ being integrable at order $k$. This property is called ``non degeneracy'' (see Definition \ref{def3}). This implies that an integrable potential is uniquely determined by its $k$-th order series expansion at $c$.
\item We then prove that if $V$ has eigenvalue $0$, the non-degeneracy property holds for $k\geq 2$, thus proving that $V=1/q_1$ is the only integrable potential in this case (section \ref{eigen0}).
\item We prove also that if $V$ has eigenvalue $2$, the non-degeneracy property holds for $k\geq 3$. This implies that an integrable potential with eigenvalue $2$ is uniquely determined by its third order series expansion at $c$. For each third order series expansion at $c$, we find an integrable potential $V$ having such a series expansion at $c$, thus proving there are no other integrable potentials of this type (section \ref{eigen2}).
\end{itemize}

\subsection{First order variational equations}

At this point, we proved that after reduction, a meromorphic homogeneous potential on $\Omega$ possessing a non-degenerate Darboux point, can be assumed to have the following properties
\begin{itemize}
\item There exists a point of the form $c=(1,0,\dots)$ which is a Darboux point of $V$ with multiplier $-1$ and $V(c)=1$.
\item The spectrum of the Hessian matrix of $V$ at $c$ is of the form $\Sp(\nabla^2 V(c))=\{2,\lambda\}$.
\item The first order variational equation near a homothetic orbit with energy $E=1$ is given by (after variable change $k_0\dot{\phi}\phi^{k_0-1}/\sqrt{2} \longrightarrow t$) 
\begin{equation}\label{eq1}
\frac{1}{2}(t^2-1) \ddot{y} +2t\dot{y} -\frac{1}{2}(n-1)(n+2) y=0 \qquad n\in\mathbb{N}
\end{equation}
\item If the first order variational equation has a Galois group whose identity component is Abelian, then
$$\lambda\in\{\textstyle{\frac{1}{2}}(n-1)(n+2),\; n\in\mathbb{N}\}$$
\end{itemize}

We now recall some properties of the solutions of the first order variational equation \eqref{eq1}. A basis of solutions can be written using a hypergeometric function ${}_2 F_1$
$${}_2 F_1(1+(n/2, 1/2-n/2,1/2,t^2),\;\; {}_2 F_1(1-n/2,3/2+n/2,3/2,t^2)t$$
These solutions can also be rewritten in terms of Legendre functions
$$(t^2-1)^{-3/4}L_P\left(\frac{1}{2},\frac{1}{2}\sqrt{8\lambda+9},\frac{t}{\sqrt{t^2-1}}\right),(t^2-1)^{-3/4}L_Q\left(\frac{1}{2},\frac{1}{2}\sqrt{8\lambda+9},\frac{t}{\sqrt{t^2-1}}\right)$$
Several properties on hypergeometric functions and their specializations can be found in \cite{42,43}.
For all $n\in\mathbb{N}^*$, one of these two functions is a polynomial as the hypergeometric series is finite. So we can build a basis of solutions given by $(P_n,Q_n)$ where $P_n$ is polynomial in $t$ of degree $n-1$ (which are related to Gegenbauer polynomials, as given in equation 158 of \cite{43}) and $Q_n$ is given by
$$Q_n(t)=P_n(t)\int \frac{1}{(t^2-1)^2P_n(t)^2} dt$$
The functions $Q_n$ are multivalued. The case $n=0$ will be a special case (see \cite{18}), as the Galois group of \eqref{eq1} will be $\{Id\}$ instead of $\mathbb{C}$. These properties are analogous to the ones for the Legendre polynomials in \cite{42}. These can also be reproved using the holonomic package of Mathematica in \cite{36}.

\medskip
The polynomials $P_n$ can be computed using the ``Rodrigues'' type formula
$$P_n(t)=\frac{1}{t^2-1}\frac{\partial^{n-1}}{\partial t^{n-1}} (t^2-1)^n$$
which gives a normalization for the leading term of $P_n$ that we will choose now and the functions $Q_n$ can be written as
$$Q_n(t)=\epsilon_n P_n(t) \arctanh\left(\frac{1}{t}\right)+\frac{W_n(t)}{t^2-1}$$
with $W_n$ being polynomials, and $\epsilon_n$ a real sequence given by
$$\epsilon_n=\frac{4^{-n}n(n+1)}{n!^2}$$

\begin{lem}\label{lemnonint} %(already proved in \cite{18})
Let $F\in\mathbb{C}(z_1)\left[z_2 \right]$ and
$$f(t)=F\left(t,\arctanh\left(\frac{1}{t}\right)\right) \in \mathbb{C}(t)\left[\arctanh\left(\frac{1}{t}\right)\right]$$
We consider the following differential field extension and its differential Galois group
$$K=\mathbb{C}\left(t,\arctanh\left(\frac{1}{t}\right),\int f dt \right), \qquad G=\hbox{Gal}_{\hbox{diff}}(K/\mathbb{C}(t))$$
If $G$ is Abelian, then
$$\frac{\partial}{\partial \alpha}  \underset{{t=\infty}}{\hbox{Res}}\;\; F\left(t,\arctanh\left(\frac{1}{t}\right)+\alpha\right) =0\quad \forall \alpha\in\mathbb{C}$$
where Res corresponds to the residue.
\end{lem}

\begin{proof}
We first remark that the Zariski closure of the monodromy group of $f$ in the complex plane $\mathbb{C}\setminus \{-1,1\}$ is exactly the Galois group $G$ because $f$ satisfies a linear differential equation whose singularities are regular. We now consider two paths: the ``eight'' path~$\sigma_1$ around the singularities $-1$ and $1$, and the path~$\sigma_2$ around infinity. At infinity, the function $F\left(t,\arctanh\left(\frac{1}{t}\right)+\alpha\right)$ will have a series expansion of the kind
$$\int F\left(t,\arctanh\left(\frac{1}{t}\right)+\alpha\right) dt=\sum\limits_{n=n_0}^{\infty} a_n(\alpha)t^n+r(\alpha)\ln\; t$$
because the function $\arctanh\left(\frac{1}{t}\right)$ has a regular point at infinity. Let us now consider the monodromy commutator
$$\sigma_l=[\sigma_2,\sigma_1^l]=\sigma_2^{-1}\sigma_1^{-l}\sigma_2\sigma_1^{l} \quad \hbox{with} \quad l\in \mathbb{Z}$$
Computing the monodromy, we obtain $\sigma_1^{l}(f)=F\left(t,\arctanh\left(\frac{1}{t}\right)+2i\pi l\right)$ and $\sigma_2(\ln t)=\ln t +2i\pi$. We deduce that
$$\sigma_l(f)=f+r(2i\pi l)-r(0)$$
This $r(2i\pi l)$ corresponds to the residue of $F\left(t,\arctanh\left(\frac{1}{t}\right)+2i\pi l\right)$ at infinity. If $G$ is Abelian, then the monodromy is commutative, and then the commutator $\sigma_l$ should act trivially on $f$. This is the case only if $r(2i\pi l)=r(0)\;\;\forall l \in \mathbb{Z}$. The function $r$ is a polynomial in $l$, so $r(2i\pi l)-r(0),\;\;\forall l\in\mathbb{C}$. This gives us the formula of the lemma
$$\frac{\partial}{\partial \alpha}  \underset{{t=\infty}}{\hbox{Res}}\;\; F\left(t,\arctanh\left(\frac{1}{t}\right)+\alpha\right) =0\quad \forall \alpha\in\mathbb{C}$$
\end{proof}

\subsection{Higher order variational equations}\label{secHig}

In the particular case of a Hamiltonian system coming from a potential with $2$ degrees of freedom, we will be able to rewrite the higher order variational equations in a simpler. We also assume that $V$ has a non degenerated Darboux point at $c=(1,0)$ with multiplier $-1$.

Looking at variational equation of order $k$ of Morales-Ramis-Simo \cite{19} page 860 (section \ref{Int}), we see that the last equation always has the following structure. There is a homogeneous part $\dot{\varphi}_t^{(k)}=X_H^{(1)}\varphi_t^{(k)}$, and non homogeneous terms involving functions already computed when solving lower order variational equations. So this last equation can be considered as a non homogeneous linear  equation.

The $X_H$ is the Hamiltonian field, and we may write $\varphi_t^{(k)}=(Y_1,Y_2,X_1,X_2)$ (we are in dimension $4$). The $X_1$ corresponds to a perturbation tangential to the homothetic orbit, and $X_2$ normal to this orbit (and $Y_1,Y_2$ are the velocities in these directions). We see also that this variational equation is not linear. But for example at order $3$, instead of considering non linear terms like $(\varphi_t^{(1)})^3$, we replace it by solutions of the symmetric power of the equation satisfied by $\varphi_t^{(1)}$ (for this term, this gives the third symmetric power of the first order variational equation, see \cite{41}).

Computing variational equations up to order $k$ will produce monomials in the components of vectors $\varphi_t^{(1)},\dots,\varphi_t^{(k)}$. Equation $(13)$ of \cite{19} can be rewritten
$$\dot{\varphi}_t^{(k)}= \sum\limits_{j=1}^k \sum \frac{j!}{m_1! \dots m_s!} X_H^{(j)}((\varphi_t^{(i_1)})^{m_1},\dots,(\varphi_t^{(i_s)})^{m_s})$$
For each fixed $j$, the inner sum is a sum monomials of the form
\begin{equation}\label{mono}
(\varphi_t^{(1)})_{w_1}^{j_1}\dots (\varphi_t^{(k)})_{w_k}^{j_k}
\end{equation}
where $w$ indicates the component of vectors $\varphi_t$. Instead of computing $\varphi_t^{(i)}$, we compute directly these monomials. We note $y_{n_1,n_2,n_3,n_4}$ the sum over all monomials \eqref{mono} having exactly $n_1$ terms with $w=1$, $n_2$ terms with $w=2$, etc. Due to symmetries of higher variational equations, considering these $y_{n_1,n_2,n_3,n_4}$ are sufficient to analyse the variational equation (meaning that the derivatives of $y$ only involve $y$). This process has also linearised the variational equation as $y_{n_1,n_2,n_3,n_4}$ corresponds to the monomials in the sum themselves. Building linear differential equations for the $y_{n_1,n_2,n_3,n_4}$ necessitates to compute the symmetric product of differential systems (as done in \cite{38}), as we need to build linear differential system satisfied by monomials of the form \eqref{mono}. At order $k$, the variational equation now writes (the last equation)
\begin{equation}\label{nondeg}
\left(\begin{array}{c}\ddot{y}_{0,0,1,0}\\ \ddot{y}_{0,0,0,1} \end{array}\right)=
\frac{1}{\phi^{3k_0}} \left(\begin{array}{c}2y_{0,0,1,0}\\ \lambda y_{0,0,0,1} \end{array}\right) +
\left(\begin{array}{c}
\sum \limits_{i=2}^{k} \frac{1}{\phi^{(i+2)k_0}}\sum\limits_{j=0}^i \frac{d_{i,j}}{(i-j)!j!} y_{0,0,i-j,j}\\ \sum \limits_{i=2}^{k} \frac{1}{\phi^{(i+2)k_0}}\sum\limits_{j=0}^i \frac{d_{i,j+1}}{(i-j)!j!} y_{0,0,i-j,j} \\ \end{array}\right)
\end{equation}
where $y_{i,0,j,0}$ satisfy differential equations corresponding to lower order variational equations. The coefficients $d_{i,j}$ are given by
$$d_{i,j}=\frac{\partial^{i+1}}{\partial q_1^{i-j+1} \partial q_2^{j}} V(c)$$

A visual process to build these differential systems is to see $y_{n_1,n_2,n_3,n_4}$ as $\dot{X_1}^{n_1}\dot{X_2}^{n_2} X_1^{n_3} X_2^{n_4}$. We differentiate this expression and simplify it using the relation
\begin{equation}\label{nondeg2}
\ddot{X}=\frac{1}{\phi^{3k_0}} \left(\begin{array}{cc}2&0\\ 0&\lambda \end{array}\right)  X +\left(\begin{array}{c}
\sum \limits_{i=2}^{k} \frac{1}{\phi^{(i+2)k_0}}\sum\limits_{j=0}^i \frac{d_{i,j}}{(i-j)!j!} X_1^{i-j}X_2^{j}\\ \sum \limits_{i=2}^{k} \frac{1}{\phi^{(i+2)k_0}}\sum\limits_{j=0}^i \frac{d_{i,j+1}}{(i-j)!j!} X_1^{i-j}X_2^{j} \\ \end{array}\right)
\end{equation}
We suppress terms degree $>k$ that could appear, and then we formally replace back the $\dot{X_1}^{n_1}\dot{X_2}^{n_2} X_1^{n_3} X_2^{n_4}$ by $y_{n_1,n_2,n_3,n_4}$.

\begin{rem}\label{remhom}
Using the Euler relation for homogeneous function
$$q_1 \partial_{q_1}V+q_2 \partial_{q_2}V=-V$$
and derivating it in $q_1$ or $q_2$ enough times at $(q_1,q_2)=(1,0)$, we obtain the relations
$$\partial_{q_1^i q_2^j}V + \partial_{q_1^i q_2^j}V +\partial_{q_1^{i+1} q_2^j}V=-\partial_{q_1^i q_2^j}V,\quad i\geq 1,j\geq 0$$
This gives all derivatives $d_{k,j}$ of $V$ of order $k+1$ as functions of lower order ones except $d_{k,k+1}$.
\end{rem}

By construction, the differential equations for the $y_{n_1,n_2,n_3,n_4}$ have a special structure. In particular, the expression of $\dot{y}_{n_1,n_2,n_3,n_4}$ only involves terms with higher or equal sum of indexes. Thus, in particular, the differential equation for $y_{n_1,n_2,n_3,n_4},\;n_1+n_2+n_3+n_4=k$ is linear homogeneous and correspond to the $k$-th symmetric power of the first order variational equation. So the $y_{n_1,n_2,n_3,n_4},\;n_1+n_2+n_3+n_4=k$ are linear combinations of product of degree $k$ of solutions of the first order variational equation, which will be in our case after a change of variable products of solutions of the first variational equations $P,Q$.

Let us now look at equation \eqref{nondeg}. This is a non homogeneous linear equation, so once we have found the expression of the non homogeneous term, we can solve it using variation of parameters. Remark also that the highest order derivatives $d_{k,k+1}$ of $V$ at $c$ only appears in this equation and not in the lower order ones. We write the solution of the second equation of~\eqref{nondeg} (after the variable change $k_0\dot{\phi}\phi^{k_0-1}/\sqrt{2} \longrightarrow t$)
\begin{equation}\label{solnon}
y(t)=y_{hom}(t)+y_{part_1}(t)+d_{k,k+1} y_{part_2}(t)
\end{equation}
isolating the term in $d_{k,k+1}$. The part $y_{hom}(t)$ is a solution of the homogeneous part, the solution $y_{part_1}(t)$ is a particular solution of equation \eqref{nondeg} without the term in $d_{k,k+1}$ and $y_{part_2}(t)$ is a particular solution of equation \eqref{nondeg} where all non homogeneous terms are removed except the one in in $d_{k,k+1}$. Let us apply a monodromy commutator
$$\sigma_{l}=[\sigma_2,\sigma_1^l]=\sigma_2^{-1}\sigma_1^{-l}\sigma_2\sigma_1^{l}$$
(with the same notation as in the proof of Lemma \ref{lemnonint}). This gives
$$\sigma_{l}(y)=\sigma_{l}(y_{hom})+\sigma_{l}(y_{part_1})+d_{k,k+1}\sigma_{l}(y_{part_2})$$
Now let us look at $y_{part_2}$. If $\lambda=\textstyle{\frac{1}{2}} (n-1)(n+2),\;n\in\mathbb{N}$, it can be computed and one solution is
$$y_{part_2}(t)=\int (t^2-1)^k Q_n^{k+1} dt$$
We now apply lemma \ref{lemnonint} which says that the monodromy element $\sigma_l$ add to such function the constant
$$G(\alpha)=\underset{{t=\infty}}{\hbox{Res}}\;\;(t^2-1)^k(Q_n+\epsilon_n\alpha P_n)^{k+1}-\underset{{t=\infty}}{\hbox{Res}}\;\;(t^2-1)^kQ_n^{k+1}$$
The equation $\sigma_{l}(y)=y$ becomes
\begin{align*}
\sigma_{l}(y)-y=\\
\sigma_{l}(y_{hom})+\sigma_{l}(y_{part_1})+d_{k,k+1}\sigma_{l}(y_{part_2})-y_{hom}-y_{part_1}-d_{k,k+1} y_{part_2}=\\
\sigma_{l}(y_{hom})+\sigma_{l}(y_{part_1})-y_{hom}-y_{part_1}+d_{k,k+1}(\sigma_{l}(y_{part_2})-y_{part_2})=\\
\sigma_{l}(y_{hom})+\sigma_{l}(y_{part_1})-y_{hom}-y_{part_1}+d_{k,k+1}G(2i\pi l)=0\\
\end{align*}
This equality is valid for any $l\in\mathbb{Z}$. This produces a system of affine equations in $d_{k,k+1}$. If the function $G(2i\pi l)$ is not zero, then this system of equations has at most one solution in $d_{k,k+1}$. This motivates the following definition

\begin{defi}\label{def3}
Let $V$ be a holomorphic homogeneous potential on $\Omega$ of degree $-1$ admitting a Darboux point of the form $c=(1,0,\dots)$ with multiplier $-1$. We note $\Sp(\nabla^2 V(c))=\{2,1/2(n-1)(n+2)\}$. Let $k\in\mathbb{N}^*$ be fixed and $(VE_k)$ the $k$-th order variational equation near the homothetic orbit associated to $c$. We assume $(VE_{k-1})$ integrable, so the identity component of the Galois group of $(VE_{k-1})$ is Abelian. We say that \emph{the integrability constraint of $(VE_k)$ is non degenerate} if
$$\frac{\partial}{\partial \alpha} \underset{{t=\infty}}{\hbox{Res}}\;\;(t^2-1)^k(Q_n+\epsilon_n\alpha P_n)^{k+1}\neq 0$$
\end{defi}

In other words, the $k$-th variational equation is seen as a system of differential equations depending on parameters. We search to understand how the Galois group of $(VE_k)$ varies with respect to the parameters. The parameter $d_{k,k+1}$ has a very special property, as it appears only one time in the system and the solutions of the system are affine functions in $d_{k,k+1}$. If the above derivative is non zero, we have that the Galois group depends explicitly on $d_{k,k+1}$, and that there is at most one value of $d_{k,k+1}$ such that the Galois group of $(VE_k)$ is Abelian.

\subsection{A rigidity result}

We will now prove that non-degeneracy implies an important rigidity property. If we take two integrable potentials ``close'' enough (meaning that enough derivatives on some Darboux point are equal), then they should be equal as proved in Lemma \ref{thm4} below.

\begin{lem}\label{thm4}
Let $V_1,V_2$ be two integrable holomorphic homogeneous potentials on $\Omega$ of degree $-1$ with a Darboux point of the form $c=(1,0,\dots)\in\Omega$ with multiplier $-1$. Assume there exists $k_0\geq 2$ such that integrability constraint of $(VE_k)$ is non degenerate $\forall k\geq k_0$. If
$$\frac{\partial^{i+j}}{\partial q_1^i\partial q_2^j} V_1(c)=\frac{\partial^{i+j}}{\partial q_1^i\partial q_2^j} V_2(c) \quad \forall (i,j) \hbox{ such that } i+j\leq k_0$$
then $V_1=V_2$.
\end{lem}

\begin{proof}
We prove this by induction. Assume that all derivatives of $V_1$ and $V_2$ are equal up to order $k \geq k_0$. Let us prove that the derivatives of order $k+1$ coincide. Using remark \ref{remhom} p \pageref{remhom}, we already know that they coincide except maybe the derivatives
$$d^{(1)}_{k,k+1}=\frac{\partial^{k+1}}{\partial q_2^{k+1}} V_1(c) \qquad d^{(2)}_{k,k+1}=\frac{\partial^{k+1}}{\partial q_2^{k+1}} V_2(c) $$
At first order, the variational equation near the homothetic orbit associated to $c$ has a Galois group whose identity component is Abelian (for $V_1$ and $V_2$). The eigenvalue $\lambda=\textstyle{\frac{1}{2}} (n-1)(n+2)$ is the same for $V_1,V_2$ as they coincide at least up to order $2$. We can then write a solution $X$ of the variational equation \eqref{solnon} at order $k$ under the form as in section \ref{secHig} page \pageref{secHig}
\begin{align}\label{mondro}
\sigma_{\alpha}(y^{(1)})=\sigma_{\alpha}(y_{hom})+\sigma_{\alpha}(y_{part_1})+d^{(1)}_{k,k+1}\sigma_{\alpha}(y_{part_2})\\
\sigma_{\alpha}(y^{(2)})=\sigma_{\alpha}(y_{hom})+\sigma_{\alpha}(y_{part_1})+d^{(2)}_{k,k+1}\sigma_{\alpha}(y_{part_2})
\end{align}
for respectively $V_1,V_2$ with
$$y_{part_2}(t)=\int (t^2-1)^k Q_n^{k+1} dt$$
The parts $y_{hom}$ and $y_{part_1}$ can be chosen to be equal as all derivatives of $V_1$ and $V_2$ are equal up to order $k$, and $y_{part_2}(t)$ can be chosen the same as the two potentials $V_1,V_2$ have the same eigenvalue $\lambda$. As $V_1,V_2$ are both meromorphically integrable, applying the monodromy commutator as before we should obtain
$$\sigma_{l}(y^{(i)})-y^{(i)}=0 ,\;\;l\in \mathbb{Z},\;\;i=1,2$$
Subtracting these two relations, we get
$$(d^{(1)}_{k,k+1}-d^{(2)}_{k,k+1}) G(2i\pi l)=0$$
Now as the integrability constraint of $(VE_{k})$ is non degenerate, we have $G(2i\pi l)\neq 0$ and thus $d^{(1)}_{k,k+1}=d^{(2)}_{k,k+1}$. Thus all derivatives of $V_1,V_2$ at $c$ up to order $k+1$ coincide. Knowing that $V_1$ and $V_2$ are holomorphic on $\Omega$, then they are equal.
\end{proof}

Lemma \ref{thm4} allows to prove uniqueness theorems: if the non degeneracy condition of Definition \ref{def3} is satisfied, then a meromorphically integrable potential is completely determined using its derivatives $c$ up to order $k$. So we now need to look in the literature for meromorphically integrable homogeneous potentials of degree $-1$ with a Darboux point of the form $c=(1,0,\dots)$ with multiplier $-1$. The space of series expansion of order $k$ of homogeneous potentials of degree $-1$ at $c$ and a fixed eigenvalue $\lambda$ is an affine space $\mathcal{E}$. If all series expansions in $\mathcal{E}$ are series expansions of meromorphically integrable potentials, then this proves that no other exist (if two meromorphically homogeneous potentials coincide up to order $k$, they are equal).

For this problem, direct search (e.g. Hietarinta's work in \cite{7}) helps a lot. Still if not enough integrable potentials are found, we only proved that the set of meromorphically integrable potentials is included inside an affine space whose dimension is bounded by $\hbox{dim}(\mathcal{E})$. Remark that this procedure is non constructive as it never allows to find new integrable potentials, but only proves at best that all of them are already found (we could still guess them through their series expansion, but due to computer limitations, we often obtain less than $10$ terms).

We now study the cases $\lambda=-1,0,2$ because these are the ones for which we know integrable potentials.

\subsection{Application to the eigenvalue $0$}\label{eigen0}

\begin{lem}\label{thm5}
Let $V$ be a holomorphic homogeneous potential on $\Omega$ of degree $-1$ such that there exists a Darboux point of the form $c=(1,0,\dots)$ with multiplier $-1$. Assume that $\Sp(\nabla^2V(c))=\{2,0\}$ and that $V$ is meromorphically integrable. Then $V=1/q_1$.
\end{lem}

\begin{proof} We just have to use Lemma \ref{thm4}. Let us first check the non degeneracy property. The functions $P_1,Q_1$ for the eigenvalue $0$ are the following
$$P_1=1\qquad Q_1=\arctanh\left( \frac{1}{t} \right)-\frac{t}{t^2-1}$$
We need to look at the following residue
$$ \underset{{t=\infty}}{\hbox{Res}}\;\; (t^2-1)^{k+1}  \left(\arctanh\left( \frac{1}{t} \right)+\alpha-\frac{t}{t^2-1} \right)^{k+2} $$
and it should be independent of $\alpha$. The easiest coefficient to study (and non trivial) seems to be the coefficient in $\alpha^{k+1}$. Denoting it $S_k$, we find after simplification
$$S_k= \underset{{t=\infty}}{\hbox{Res}}\;\;  \left( t^2-1 \right) ^{k+1} \left( \left( k+2\right)\arctanh \left( \frac{1}{t} \right)-\frac {kt+2t}{t^2-1}\right) $$
By expanding, we remark that the second term always gives a zero residue. Indeed, in the expansion, the fraction simplifies and we get a polynomial. Then, we will compute
$$S_k= \underset{{t=\infty}}{\hbox{Res}}\;\;  (k+2)\left( t^2-1 \right) ^{k+1} \arctanh \left( \frac{1}{t} \right)= \frac{k+2}{2}\int\limits_{-1}^1 \left( {t}^{2}-1 \right) ^{k+1}dt >0$$
The last equality is produced with the expansion of $\arctanh \left( \frac{1}{t} \right)$ at infinity. We deduce that
$$S_k\neq 0 \quad\forall k\geq 1$$

Using Lemma \ref{thm4}, we now know that there is a unique potential with $(1,0,\dots)$ as Darboux point with multiplier $-1$ and eigenvalue $0$. The potential $1/q_1$ satisfies these conditions, and is integrable because it is invariant by translation.
\end{proof}

\textbf{Conclusion}. After rotation, an integrable potential $V$ with a zero eigenvalue near a non degenerate Darboux point corresponds to the potential
$$V=\frac{a}{q_1},\qquad a\in\mathbb{C}^*$$

\subsection{Application to the eigenvalue $2$}\label{eigen2}

In the case of the eigenvalue $2$, we use again the same method. First we will prove that the integrability constraint of $(VE_k)$ is non degenerate at order $k\geq 3$. Thus an integrable potential is uniquely defined by its first three derivatives. In Lemma \ref{thm7}, we find integrable potentials for all possible series expansions, including an exceptional case that appears to coincide after rotation with the ``Hietarinta'' potential
\begin{equation}\label{eqhieta}
\frac{aq_1}{(q_1+\epsilon iq_2)^2}
\end{equation}

Remark now that in the case of the eigenvalue $2$, the Hessian matrix should be diagonalizable for meromorphic integrability. This is a constraint for integrability of first order variational equation, and a complete analysis of the non diagonalizable case at order $1$ is given by Duval and Maciejewski in \cite{30}. Let us first prove the non degeneracy hypothesis of Lemma \ref{thm4}.

\begin{lem}\label{thm6}
Let $V$ be a holomorphic homogeneous potential on $\Omega$ of degree $-1$ such that there exists a Darboux point of the form $c=(1,0,\dots)$ with multiplier $-1$. Assume that $\Sp(\nabla^2V(c))=\{2,2\}$. Then the integrability constraint of the $k$-th order variational equation $(VE_k)$ is non degenerate for $k\geq 3$, and degenerate at order $k=2$.
\end{lem}

\begin{proof} We need to look at the following residue
$$ \underset{{t=\infty}}{\hbox{Res}}\;\;  (t^2-1)^{k+1} \left(-\frac{6t^2-4}{t^2-1}+6 t\alpha+6 t \arctanh \left( \frac{1}{t} \right)\right)^{k+2}$$
and this residue should be independent of $\alpha$ to prove non degeneracy. We will look at
$$S_k^{(1)}:= \underset{{t=\infty}}{\hbox{Res}}\;\;  (t^2-1)^{k+1} t^{k+1} \left(-\frac{6t^2-4}{t^2-1}+6 t \arctanh \left( \frac{1}{t} \right)\right)$$
which corresponds to the coefficient in $\alpha^{k+1}$ (after simplifying a non zero factor). Still we will see that studing this sequence is not enough, as it is not always non-zero. We will also need to look at another one
$$S_k^{(2)}:= \underset{{t=\infty}}{\hbox{Res}}\;\;  (t^2-1)^{k+1}t^k \left(-\frac{6t^2-4}{t^2-1}+6 t \arctanh \left( \frac{1}{t} \right)\right)^2$$
Then, we want to prove
$$S_k^{(1)}\neq 0 \hbox{ or } S_k^{(2)}\neq 0 \qquad \forall k\geq 2$$
More precisely, we will prove that
$$S_{2k}^{(1)}\neq 0 \hbox{ and } S_{2k+1}^{(2)}\neq 0 \qquad \forall k\geq 1$$

These sequences are $D$-finite, and as such recurrence formulas can be automatically found and proved for these sequences \cite{34,35,36,37}. Following this creative telescoping approach, we found using these algorithms (either Mgfun for Maple, or holonomics for Mathematica) the following recurrences for $S_{2n}^{(1)}$
\begin{align*}
64(2n+3)(2n+1)(6n+11)(n+1)^2f(n)-\\
(20736n^5+152064n^4+439200n^3+622752n^2-431784n-116328)\\
f(n+1)+36(6n+5)(3n+5)  \left( 3n+4 \right)  \left( 6n+13 \right) \left( 6n+17 \right) f(n+2) 
\end{align*}
This recurrence can be solved explicitly and gives the formula
$$S_{2n}^{(1)}=-\frac{\pi\Gamma \left(2n+2 \right)27^{-n}}{72 \Gamma \left(n+\frac{7}{6} \right)\Gamma \left(n+\frac{11}{6} \right)}$$
This expression never vanishes. We do the same for $S_{2n+1}^{(2)}$. We find a third order recurrence and solve it
\begin{footnotesize}
$$S_{2n+1}^{(2)}=-\frac{\pi 27^{-n}\Gamma\left(2n+3 \right)}{3456\Gamma\left( n+\frac{7}{3}\right)\Gamma\left( n+\frac{5}{3} \right)}  
\sum _{{\it k}=0}^{n-1} \left(\frac{\left( 3{\it k}+4 \right) \Gamma  \left( {\it k}+5/3 \right) \Gamma  \left( {\it k}+7/3 \right)}{  \left( {\it k}+1 \right)  \left( {\it k}+2\right)  \left( 2{\it k}+3 \right)\Gamma\left( {\it k}+{\frac {13}{6}} \right)\Gamma \left( {\it k}+{\frac {11}{6}} \right)   } \right)$$
\end{footnotesize}
Using this expression, we find that $S_{2n+1}^{(2)}$ never vanish for $n\geq 1$. This proves the non degeneracy condition for order $\geq 3$. At order $2$, the two formulas vanish. Since in this case the residue is a polynomial in $\alpha$ of degree at most $2$, this implies that the residue is constant. So the $\alpha$ derivative is zero and the integrability constraint is degenerate.
\end{proof}

We now need to find integrable homogeneous potentials of degree $-1$ which admit a Darboux point $c$ with spectrum $\{2,2\}$. We already know the potential
$$\frac{a}{q_1}+\frac{b}{q_2},\qquad a,b\in\mathbb{C}^*$$
which is integrable. Computation gives that Darboux points have the eigenvalue $2$. So we need to prove that after rotation, all possible $3$-rd order derivatives can be produced. As shown below, an exceptional case will be found and will correspond to the Hietarinta potential \eqref{eqhieta}.

\begin{lem}\label{thm7}
Let $V$ be a holomorphic homogeneous potential on $\Omega$ of degree $-1$ such that there exists a Darboux point of the form $c=(1,0,\dots)$ with multiplier $-1$ and $\nabla^2 V(c)=2 I_2$. Then it corresponds after rotation to a potential of the form
\begin{equation}\label{eq9}
\frac{a}{q_1}+\frac{b}{q_2},\qquad a,b\in\mathbb{C}^*
\end{equation}
except if $V$ admits a series expansion in $q$ at $c$ of the following form
$$V(c+q)=1-q_1+(q_1^2+q_2^2)+dq_1^3+3dq_1q_2^2 \pm 2id q_2^3+o(q^3)$$
for which $V$ corresponds after rotation to the Hietarinta potential \eqref{eqhieta}.
\end{lem}

\begin{proof}
We expand $V$ at $c=(1,0,\dots)$ which gives
$$V(c+q)=V(c)-q_1+(q_1^2+q_2^2)+aq_1^3+bq_1^2q_2+cq_1q_2^2+dq_2^3+o(q^3)$$
Using remark \ref{remhom} page \pageref{remhom}, we obtain the following values
$$\partial_{1,1,1}V(c)=-6\quad \partial_{1,1,2}V(c)=0\quad \partial_{1,2,2}V(c)=-6$$
Then the series expansion of $V$ on $c$ is always of the form
$$V(c+q)=1-q_1+(q_1^2+q_2^2)-(q_1^3+3q_1q_2^2+dq_2^3)+o(q^3)$$
where $d\in\mathbb{C}$. We want now prove that such an expansion can correspond to the expansion of the potential \eqref{eq9} after rotation. So we will make a rotation of the coordinates $q_1,q_2$. After rotation, the potentials \eqref{eq9} can be written
$$\frac{a}{\mathfrak{c}q_1-\mathfrak{s}q_2}+\frac{b}{\mathfrak{s}q_1+\mathfrak{c}q_2},\qquad \mathfrak{c}^2+\mathfrak{s}^2=1,a,b\in\mathbb{C}^*$$
The condition of admitting a Darboux point at $c=(1,0)$ with multiplier $-1$ implies that this family of potentials can be written
$$V=\frac{\mathfrak{c}^3}{\mathfrak{c}q_1-\mathfrak{s}q_2}+\frac{\mathfrak{s}^3}{\mathfrak{c}q_1+\mathfrak{s}q_2},\hbox{ with } \mathfrak{c}^2+\mathfrak{s}^2=1$$
We make series expansion of this expression near $c=(1,0)$ and by identification, we get
$$\frac{-\mathfrak{c}^2+\mathfrak{s}^2}{\mathfrak{c}\mathfrak{s}}=d ,\hbox{ with } \mathfrak{c}^2+\mathfrak{s}^2=1.$$
This produces the solution
$$\mathfrak{s}=\frac{1}{\sqrt{2}}\sqrt{\frac{4+d^2+\sqrt {4d^2+d^4}}{4+d^2}}$$
which is valid for $d\neq  2i\epsilon$ with $\epsilon=\pm 1$.

For $d= 2i\epsilon$, there are no solutions, and this is the exceptional case. Let us check that the Hietarinta potential \eqref{eqhieta} corresponds to this case. We will only study the case $\epsilon=+1$ (the case $\epsilon=-1$ is exactly similar). After rotation, we get
$$V=a\frac{q_1^2+q_2^2}{(q_1+iq_2)^3}+\frac{ab}{q_1+i q_2},\qquad a,b\in\mathbb{C}^* $$
The condition of having a Darboux point at $c=(1,0)$ with multiplier $-1$ gives
$$V=-\frac{1}{2}\frac{q_1^2+q_2^2}{(q_1+iq_2)^3}+\frac{3}{2(q_1+iq_2)}$$
We compute the series expansion at $c=(1,0,\dots)$ and this gives exactly the good expansion. Using Lemma \ref{thm4},\ref{thm6}, we know that for each series expansion at order $3$, there exists at most one meromorphically integrable potential. We found a meromorphically integrable potential for any possible series expansion at order $3$, and so we found all meromorphically integrable potentials with the eigenvalue $2$.
\end{proof}

\section{Case of the eigenvalue $-1$}\label{secCas}

The case of the eigenvalue $-1$ is much more difficult because the non degeneracy hypothesis of Lemma \ref{thm4} does not hold. We need to use a completely different method. We already guess that this case will correspond to the potential $V=r^{-1}$ invariant by rotation. This potential integrates in polar coordinates, which are the action-angles coordinates for this potential. Then, to see some pattern in higher variational equations, it seems to be a good idea to compute all these higher variational equations in polar coordinates. The integrable case $V=r^{-1}$ is quite simple to describe in polar coordinates, as it is the only potential that does not depend on the angle coordinate (the coordinate $\theta$).

So we will first compute higher variational equations up to order $2$. Then we will recognize that a strong integrability constraint comes from a particular a simple perturbation, which will allows us to study only a subsystem of these higher variational equations system. We prove in particular that the $k$-th variational equation possesses invariant vector spaces; we will find one which is small enough such that the reduced system on this subspace can be more easily analyzed, and not too simple. The solutions have non-commutative monodromy which puts constraints on the derivatives of the potential. The first non-trivial integrability condition appears at order $3$ with a dilogarithmic term. At higher order, we will prove that a non zero $(k+1)$-th derivative $U^{(k+1)}(0)\neq 0$ (the potential being $V=r^{-1}U(\theta)$ in polar coordinates on a neighbourhood of $\theta=0$, eventually for a good branch choice) implies that the Picard Vessiot field of the $(2k-1)$-th variational equation contains a dilogarithmic term, and thus that the Galois group is not Abelian.

\begin{prop}\label{thm9} (proved page \pageref{secPro})
Let $V$ be a holomorphic homogeneous potential on $\Omega$ of degree $-1$ such that there exists a Darboux point of the form $c=(1,0,\dots)$ with multiplier $-1$. Assume that $\Sp(\nabla^2 V(c))=\{2,-1\}$. If $V$ is integrable, then $V =r^{-1}$.
\end{prop}

\bigskip

The strategy of the proof will be the following
\begin{itemize}
\item We consider a potential $V=r^{-1} U(\theta)$ with $U(0)=U'(0)=\dots= U^{(k)}(0)=0$ (with $k\geq 2$). We want to prove that the $(2k-1)$-th order variational equation has a non-Abelian Galois group.
\item We first find an invariant vector space $\mathcal{W}$ of the $(2k-1)$-th order variational equation (section \ref{inv}).
\item Reduced on $\mathcal{W}$, the $(2k-1)$-th order variational equation becomes more manageable, and solve it for some of the unknowns (section \ref{secPro}). We find that the solutions contain a dilogarithmic term if $U^{(k+1)}(0)\neq 0$, implying that the Galois group of the whole equation is not Abelian.
\item By induction, we conclude that for $V$ being integrable, all derivatives of $U$ at $0$ should be zero, thus concluding that $V=r^{-1}$.
\end{itemize}

\bigskip

Remark that the two following subsections are not necessary to prove Theorem \ref{thm9}. They are here to show in details why the previous approach through the non degeneracy property does not work. In particular, the Galois group of variational equations depends on the highest derivative $U^{(k+1)}(0)$, but we cannot obtain a term with a non-commutative monodromy: the most we obtain is simply an integral of a rational fraction, and thus the Galois group is either of the form $G$ or $G\times \mathbb{C}$, depending on $U^{(k+1)}(0)$. This is thus not sufficient to obtain a uniqueness theorem.

\subsection{Looking at variational equation of order $2$}

Before proving Proposition \ref{thm9}, let us first look only at order $2$. 

\begin{lem}
Let $V=r^{-1}U(\theta)$ be a holomorphic homogeneous potential on $\Omega$ of degree $-1$ such that there exists a Darboux point of the form $c=(1,0,\dots)$ with multiplier $-1$. We note locally near $\theta=0$ in polar coordinates $V(q,r,\mathbf{w}) =r^{-1}U(\theta)$ on the branch on which lies $c$. Assume that $\Sp(\nabla^2 V(c))=\{2,-1\}$. The Galois group of the second order variational equation near the homothetic orbit associated to $c$ of the Hamiltonian field in polar coordinates is $\mathbb{C}^2$ if $U^{(3)}(0)\neq 0$ and $\mathbb{C}$ if $U^{(3)}(0)= 0$.
\end{lem}

\begin{proof}
The potential $V=r^{-1} U(\theta)$ gives the following differential equations in polar coordinates
\begin{equation}\label{eqpolar}
\ddot{r}-r\dot{\theta}^2=-\frac{1}{r^2}U(\theta),\qquad \ddot{\theta}+2\frac{\dot{r}}{r}\dot{\theta}=\frac{1}{r^3}U'(\theta).
\end{equation}
Let us linearize this equation near a homothetic orbit corresponding to the critical point $0$ of $U$. We assume moreover that $U''(0)=0$, which corresponds to $\Sp(\nabla^2V(c))=\{2,-1\}$, and that $U(0)=1$ after dilatation (which implies that the multiplier is $-1$). We get at first order
\[\ddot{r}=\frac{2 U(0)}{\phi^3}r,\qquad \ddot{\theta}+2\frac{\dot{\phi}}{\phi}\dot{\theta}=\frac{U''(0)}{\phi^3} \theta=0 \quad \hbox{ with } \dot{\phi}^2/2=\phi^{-1}+1 \]
This parametrization $\phi$ is the same as in definition \ref{def22} with $k_0=1$ (it happens that for computing variational equations in polar coordinates, we no longer need to consider parametrizations with $k_0>1$). We now make the variable change $\dot{\phi}/\sqrt{2}\longrightarrow t$ which gives
\[\frac{1}{2}\, \left( {t}^{2}-1 \right) \ddot{r} +2\dot{r} t-2\,r =0,\qquad \frac{1}{2}\, \left( {t}^{2}-1 \right) \ddot{\theta} =0. \]
Of course these equations are integrable (because they correspond to an integrable case of Theorem \ref{thmmorales} and the solutions are
$$r(t)=C_1 P_2(t)+C_2Q_2(t), \qquad \theta(t)=C_3t+C_4$$
Now we take a look at second-order variational equations. As in equation \eqref{nondeg2}, we first compute the series expansion of order $2$ of equation \eqref{eqpolar} at $\dot{r}=\dot{\phi},\dot{\theta}=0,r=\phi,\theta=0$
\begin{equation}\begin{split}
\ddot{r}-\phi \dot{\theta}^2 =\frac{2}{\phi^3} r-\frac{3}{\phi^4}r^2\\
\ddot{\theta}+2\frac{\dot{\phi}}{\phi}\dot{\theta}+\frac{2}{\phi} \dot{r}\dot{\theta}-\frac{2\dot{\phi}}{\phi^2} r\dot{\theta}=\frac{1}{\phi^3} U^{(3)}(0)\theta^2\\
\end{split}\end{equation}
Using again the same procedure as in page \pageref{nondeg2}, the second order variational equation may be written (after variable change $\dot{\phi}/\sqrt{2}\longrightarrow t$)
\begin{equation}\label{eq6}\begin{split}
\frac{1}{2}(t^2-1) \ddot{r}_2 +2t\dot{r}_2-\frac{1}{2}\dot{\theta}_1^2=2r_2-3(t^2-1)r_1^2\\
\frac{1}{2}\ddot{\theta}_2+2tr_1\dot{\theta}_1+(t^2-1)\dot{r}_1\dot{\theta}_1=\frac{1}{2}\frac{1}{t^2-1}U^{(3)}(0)\theta_1^2
\end{split}\end{equation}
where $r_1,\theta_1$ are solutions of the first order variational equation. The first equation of \eqref{eq6} integrates because the non homogeneous term
$$-\frac{1}{2}\dot{\theta}_1^2=-\frac{1}{2}C_3^2$$
corresponds to a particular solution to $r_2$ of the form
$$\int (t^2-1)Q_2^3 dt$$
(where $Q_2$ is defined page \pageref{eq1}) and whose monodromy is commutative. For the second equation of \eqref{eq6}, we find
$$\frac{1}{2} \ddot{\theta}_2+2t(C_1 P_2+C_2Q_2)C_3+(t^2-1)(C_1 \dot{P}_2+C_2\dot{Q}_2)C_3=\frac{1}{2}\frac{1}{t^2-1}U^{(3)}(0)(C_3t+C_4)^2$$
The solution can be written as
$$2\iint -2t(C_1 P_2+C_2Q_2)C_3-(t^2-1)(C_1 \dot{P}_2+C_2\dot{Q}_2)C_3+\frac{1}{2}U^{(3)}(0)\frac{(C_3t+C_4)^2}{t^2-1} dt^2$$
We have the following formulas
$$P_2=4t\quad Q_2=\frac{3}{8}t \arctanh \left( \frac{1}{t} \right)+\frac {\frac{1}{4}-\frac{3}{8}t^2}{t^2-1}$$ 
The terms in $P_2$ are polynomials and integrate well. For the terms in $Q_2$, we find the following expression (up to integration constants)
\begin{align*}
\iint -2tQ_2-(t^2-1)\dot{Q}_2 dt dt= \iint \frac{3}{8}(3t^2-1)\arctanh\left( \frac{1}{t} \right)-\frac{9}{8}t dtdt=\\
\frac{3}{32}(t^2-1)^2\arctanh\left(\frac{1}{t} \right)-\frac{3}{32} t^3
\end{align*} 
Then the only term left is
\begin{equation}\label{eq7} 
\iint \frac{1}{2}U^{(3)}(0)\frac{(C_3t+C_4)^2}{t^2-1} dt dt \in\mathbb{C}\left[t,\arctanh\left( \frac{1}{t} \right),\ln\left( t^2-1 \right) \right]
\end{equation}
Then the second-order variational equation is always integrable, and moreover we have computed its Galois group
\begin{itemize}
\item If $U^{(3)}(0)\neq 0$ then the Galois group of \eqref{eq6} is $\mathbb{C}^2$
\item If $U^{(3)}(0)= 0$ then the Galois group of \eqref{eq6} is $\mathbb{C}$
\end{itemize}
\end{proof}

\subsection{Degeneracy of higher variational equations}

Let us now look at the non-degeneracy property of variational equation. In the case of eigenvalue $-1$, we have $\epsilon_0=0$ (see page \pageref{eq1}) because the first order variational equation has two independent rational solutions
$$P_0=t(t^2-1)^{-1}\qquad Q_0=(t^2-1)^{-1}$$
Still we could think that a similar condition to \ref{def3} of the non-degeneracy could still apply. The term corresponding to the highest order derivative of the potential is given by
$$\int (t^2-1)^{k} (aQ_0+bP_0)^{k+1} dt dt \in\mathbb{C}\left[t,\arctanh\left( \frac{1}{t} \right),\ln\left( t^2-1 \right) \right]$$
and thus this term has always a commutative monodromy. Computing variational equations of the Hamiltonian field in polar coordinate does not help either.

\begin{prop}\label{thm10}
Let $V$ be holomorphic homogeneous potential on $\Omega$ of degree $-1$ with a Darboux point of the form $c=(1,0,\dots)$ with multiplier $-1$ (and thus $c$ correspond to an angle $\theta=0$). Assume that $\Sp(\nabla^2 V(c))=\{2,-1\}$. Assume that $U^{(i)}(0)=0 \;\forall i=1\dots k$ then the fact that the identity component of the Galois group of the $k$-th order variational equation is Abelian or not does not depend on the value of $U^{(k+1)}(0)$.
\end{prop}

\begin{proof}
The case of order $2$ corresponds to the previous proof. Let us look now at order $k$. We pick in the equations the non homogeneous terms where $U^{(k+1)}(0)$ appear. The only equation where such term appears is the following
\begin{equation}\label{eqhighdeg}
\frac{1}{2} \ddot{\theta}_k = \frac{1}{k!}\frac{U^{(k+1)}(0)}{t^2-1}\theta_1^k
\end{equation}
where $\theta_1$ is solution of the first order variational equation (we have removed all non homogeneous terms in which $U^{(k+1)}(0)$ does not appear). The solution for $\theta_1$ is $\theta_1(t)=at+b$, and then substituting this expression, we obtain that the solution of equation \eqref{eqhighdeg} is of the form
$$\theta_k(t)=\frac{2U^{(k+1)}(0)}{k!}\iint \frac{(at+b)^k}{t^2-1} dt dt \in\mathbb{C}\left[t,\arctanh\left( \frac{1}{t} \right),\ln\left( t^2-1 \right) \right]$$
which can be checked using recursive integration by parts. This term then has a commutative monodromy, and then the fact that the identity component of the Galois group is Abelian or not does not depend on the value of $U^{(k+1)}(0)$.
\end{proof}

\begin{rem}
We remark that the integral
$$\iint \frac{(at+b)^k}{t^2-1} dt dt$$
does not belong to the Picard-Vessiot field of the first order variational equation (which is $\mathbb{C}\left[t,\arctanh\left(\textstyle{\frac{1}{t}}\right) \right]$). So the Picard-Vessiot field of the $k$-order variational equation is generically larger (when $U^{(k+1)}(0)\neq 0$), and the Galois group becomes at least $\mathbb{C}^2$. But this does not give us any integrability condition, as the Galois group could still be Abelian (with a higher dimension). This is precisely why this case is particularly difficult. We cannot use non degenerescence properties, and so we need to keep these unknown derivatives $U^{(i)}(0)$ as parameters and go higher in the order of variational equations.
\end{rem}

\subsection{An invariant subspace of the $(2k-1)$-th order variational equation}\label{inv}

Let us first look at variational equation of order ${2k-1}$. We compute the series expansion of equation \eqref{eqpolar} at $\dot{r}=\dot{\phi},\dot{\theta}=0,r=\phi,\theta=0$ of order $k\geq 3$
\begin{equation}\begin{split}\label{eqvar}
\ddot{r}-\phi \dot{\theta}^2-r\dot{\theta}^2 =\sum\limits_{i=1}^{2k-1}\frac{(-1)^{i+1}(i+1)}{\phi^{i+2}}r^i+\\ \sum\limits_{i=k+1}^{2k-1} \sum\limits_{j=0}^{i-k-1} \frac{(-1)^{j+1}(j+1)}{\phi^{j+2}(i-j)!}U^{(i-j)}(0)r^j\theta^{i-j}\\
\ddot{\theta}+\sum\limits_{i=0}^{2k-2} \frac{2(-1)^i\dot{\phi}}{\phi^{i+1}}r^i\dot{\theta}+\sum\limits_{i=0}^{2k-3}\frac{2(-1)^i}{\phi^{i+1}}\dot{r} r^i\dot{\theta} =\\
\sum\limits_{i=k}^{2k-1} \sum\limits_{j=0}^{i-k} \frac{(-1)^{j}(j+1)(j+2)}{2\phi^{j+3}(i-j)!}U^{(i-j+1)}(0)r^j\theta^{i-j} 
\end{split}\end{equation}
The $U^{(i)}(0)\;\;i=k+1\dots 2k$ are parameters. Using these series expansions, let us build now the ${2k-1}$-th variational equation. As explained in section 4.2., we use the substitution
$$y_{i,j,l,m}=\dot{r}^i\dot{\theta}^jr^l\theta^m$$
For $i+j+l+m={2k-1}$, the differential equation system for the $y_{i,j,l,m}$ is the $(2k-1)$-symmetric power of the first order variational equation. As we do not want to compute the complete solution of the ${2k-1}$-th variational equation, we will only compute one well chosen solution. To simplify the equation, we will first build an invariant vector space, and then solve the variational on this subspace (and in fact only for some of the variables).

\begin{lem}\label{leminv}
Assume that $U(0)=1,U^{(i)}(0)=0\;\;\forall i= 1\dots k$. Then the vector space $\mathcal{W}$ given by the conditions
$$y_{i,j,l,m}=0\quad \forall i,j,l,m \hbox{ such that } j+m\geq k,\; i+l\geq 1 \quad (C_1)$$
$$y_{i,j,l,m}=0\quad \forall i,j,l,m \hbox{ such that } j\geq 1,\; i+l\geq 1  \quad (C_2)$$
$$y_{i,j,l,m}=0\quad \forall i,j,l,m \hbox{ such that } j\geq 2  \quad (C_3)$$
$$y_{i,j,l,m}=0\quad \forall i,j,l,m \hbox{ such that } j\geq 1,\; j+m\geq k+1  \quad (C_4)$$
is an invariant vector space of the $2k-1$-th variational equation.
\end{lem}

Let us explain why we consider this vector space. We want to build an analogue of normal variational equation (which is properly defined only at order $1$). In particular, we want to suppress all terms in $\dot{r},r$ in the second equation of \eqref{eqvar}. This is the reason of the conditions $(C_1),(C_2)$. The conditions $(C_3),(C_4)$ are necessary to $\mathcal{W}$ to be invariant. 

The reason of conditions $(C_3),(C_4)$ is that we need to suppress the term in $\dot{\theta}^2$ corresponding to centrifugal force. Physically, this means that the perturbation we are interested in will correspond to very small values of $\dot{\theta}$, and if $\dot{\theta}$ is non-zero, $\theta$ should be small also (condition $(C_4)$). The condition $(C_2)$ implies that the Coriolis force coming from perturbations of order $\geq 2$ is negligible.

\begin{proof}
We only need to prove that the derivative in time of these $y_{i,j,l,m}=0$ only involve these $y_{i,j,l,m}$ (because then the differential equation being linear, $0$ will be solution of this subsystem). Let us differentiate $\dot{r}^i\dot{\theta}^jr^l\theta^m$. Using Leibniz differentiation rule, this produces $4$ terms
\begin{equation}\label{deriv}
m\dot{r}^i\dot{\theta}^{j+1}r^l\theta^{m-1}+l\dot{r}^{i+1}\dot{\theta}^jr^{l-1}\theta^m+i\ddot{r}\dot{r}^{i-1}\dot{\theta}^jr^l\theta^m+j\ddot{\theta}\dot{r}^i\dot{\theta}^{j-1}r^l\theta^m
\end{equation}
The two first terms still satisfy the condition, and for the two last ones we have to replace $\ddot{r},\ddot{\theta}$ using relation \eqref{eqvar}.
For $\ddot{r}$
\begin{itemize}
\item If condition $(C_1)$ is satisfied, the first two terms of the first equation \eqref{eqvar} will produce terms with degree in $\dot{\theta}$ at least $2$, and thus satisfying condition $(C_3)$. In the right part, the only potentially problematic terms in the sums (of the first equation \eqref{eqvar}) are in the second one for $j=0$
$$\sum\limits_{i=k+1}^{2k-1} -\frac{1}{\phi^2i!}U^{(i)}(0)\theta^i$$
As $j+m\geq k$, after multiplication we get terms of degree in $(\theta,\dot{\theta})\geq 2k+1$, and so are discarded because we study variational equation of order $2k-1$.
\item If condition $(C_2)$ is satisfied, the two first terms of the first equation \eqref{eqvar} will produce terms with degree in $\dot{\theta}$ at least $3$, and thus satisfying condition $(C_3)$. In the right part, the only potentially problematic terms are in the second sum (of the first equation \eqref{eqvar}) for $j=0$. These will produce terms with degree in $\theta$ at least $k+1$ and degree in $\dot{\theta}$ at least $1$, so satisfying condition $(C_4)$.
\item If the condition $(C_3)$ is satisfied, the degree in $\dot{\theta}$ cannot decrease, and so the three terms satisfy condition $(C_3)$.
\item If condition $(C_4)$ is satisfied, the degree in $\dot{\theta}$ and in $\theta$ cannot decrease, and so the three terms satisfy condition $(C_4)$.
\end{itemize}
For $\ddot{\theta}$
\begin{itemize}
\item If condition $(C_1)$ is satisfied, the two first terms of the second equation \eqref{eqvar} contains $\dot{theta}$, and thus condition $(C_1)$ is still satisfied. In the right sum, problematic terms would be those containing no $\dot{\theta},\theta$, corresponding to $i=j$, but these are out of the interval of summation.
\item If condition $(C_2)$ is satisfied, the two first terms of the second equation \eqref{eqvar} will produce terms with degree in $\dot{\theta}$ at least $1$ and degree in $r,\dot{r}$ at least $1$, and thus satisfying condition $(C_2)$. The right sum will produce terms with degree in $\theta$ at least $k$ and degree in $r,\dot{r}$ at least $1$, so satisfying condition $(C_1)$.
\item If condition $(C_3)$ is satisfied, the two first terms of the second equation \eqref{eqvar} will produce terms with degree in $\dot{\theta}$ at least $2$, thus satisfying condition $(C_3)$. The right sum produces terms of degree in $\dot{\theta}$ at least $1$ and degree in $\theta$ at least $k$, so satisfying condition $(C_4)$.
\item If condition $(C_4)$ is satisfied, the two first terms of the second equation \eqref{eqvar} will produce terms with degree in $\dot{\theta}$ at least $1$ and degree in $\theta,\dot{\theta}$ at least $k+1$, thus satisfying condition $(C_4)$. The right sum produces terms of degree in $\theta,\dot{\theta}$ at least $2k$, and thus which are suppressed because we study only the $(2k-1)$-th variational equation
\end{itemize}
So this subspace is invariant.
\end{proof}
We can now formally remove the corresponding terms in equation \eqref{eqvar}
\begin{equation}\begin{split}\label{eqvar3}
\ddot{r}= \sum\limits_{i=1}^{2k-1}\frac{(-1)^{i+1}(i+1)}{\phi^{i+2}}r^i+ \sum\limits_{i=k+1}^{2k-1} -\frac{U^{(i)}(0)}{\phi^2i!}\theta^i \\
\ddot{\theta}+2\frac{\dot{\phi}}{\phi}\dot{\theta}= \sum\limits_{i=k}^{2k-1} \frac{U^{(i+1)}(0)}{\phi^3 i!}\theta^i
\end{split}\end{equation}
As we see, in the second equation, terms containing $(\dot{r},r)$ no longer appear.

\begin{rem}
One of the interest of the invariant space $\mathcal{W}$ is that its dimension is much lower. Moreover, in the following, we will only be interested by the second equation of \eqref{eqvar3}, and thus reasoning in the dimension of $\mathcal{W}'=\mathcal{W}\cap \{y_{i,j,l,m}=0,\;\;\forall i+l\geq 1 \}$. By guessing, we find
$$\dim (VE_{2k-1})=\frac{1}{6}(2k+3)(k+4)(2k^2+11k+17) \qquad \dim \mathcal{W}'=3k-1$$
$$\dim \mathcal{W}=4 \prod\limits_{s=0}^{k} \left(\frac{7s^3+51s^2+134s+114}{7s^3+30s^2+53s+24}\right) \sim  \frac{7}{6} k^3$$
Clearly, studying the second equation of \eqref{eqvar3} on $\mathcal{W}$ and computing solutions for variables in $\mathcal{W}'$ will be much easier than on the complete system.
\end{rem}

\subsection{Proof of Proposition \ref{thm9}}\label{secPro}

\begin{proof}

To prove Proposition \ref{thm9}, it is only necessary to prove that  $U^{(i)}(0)=0\;\;\forall i\in\mathbb{N}^*$. It is already proved for $i=1,2$ by hypothesis. Let us prove this by recurrence. Assume that $U^{(i)}(0)=0\;\;\forall i= 1\dots k$. We want to prove that $U^{(k+1)}(0)=0$.

Let us now study the $(2k-1)$-th order variational equation, and in particular on the invariant subspace $\mathcal{W}$ given by Lemma \ref{leminv}. We may now try to find a solution of the variational equation on this invariant subspace. We will only compute closed form expression for some of the unknowns (those who appear in the second equation of \eqref{eqvar3}). We have
$$\dot{y}_{0,0,0,m}=m y_{0,1,0,m-1}=0 \qquad \forall m\geq k+1$$
We choose the solution $y_{0,0,0,m}=0,\; m=k+1\dots 2k-2$ and $y_{0,0,0,2k-1}=1$. We also find the differential equation
$$\dot{y}_{0,0,0,k}=k y_{0,1,0,k-1} \qquad \dot{y}_{0,1,0,k-1}=-\frac{2\dot{\phi}}{\phi}y_{0,1,0,k-1}+\frac{U^{(k+1)}(0)}{\phi^3 k!}y_{0,0,0,2k-1}$$
Substituting $y_{0,0,0,2k-1}$ by its expression, we get
$$\ddot{y}_{0,0,0,k}+\frac{2\dot{\phi}}{\phi}\dot{y}_{0,0,0,k}=\frac{U^{(k+1)}(0)}{\phi^3 (k-1)!}$$
The other interesting equation of the variational equations is
$$\ddot{y}_{0,0,0,1}+\frac{2\dot{\phi}}{\phi}\dot{y}_{0,0,0,1}= \frac{U^{(k+1)}(0)}{\phi^3 k!}y_{0,0,0,k}+\frac{U^{(2k)}(0)}{\phi^3 (2k-1)!}  $$
We now make the variable change $\dot{\phi}/\sqrt{2} \longrightarrow t$. This produces the equations
\begin{equation}\begin{split}
\frac{1}{2}(t^2-1)\ddot{y}_{0,0,0,k}=\frac{U^{(k+1)}(0)}{(k-1)!}\\
\frac{1}{2}(t^2-1)\ddot{y}_{0,0,0,1}= \frac{U^{(k+1)}(0)}{ k!}y_{0,0,0,k}+\frac{U^{(2k)}(0)}{(2k-1)!}
\end{split}\end{equation}
We can now solve them. We find $y_{0,0,0,k}=$
\begin{align*}
\iint \frac{2U^{(k+1)}(0)}{(k-1)!(t^2-1)} dtdt=
-\frac{2U^{(k+1)}(0)}{(k-1)!}\left(t\arctanh\left(\frac{1}{t}\right)+\frac{1}{2}\ln\left( t^2-1 \right)\right)
\end{align*}
and then
\begin{align*}
y_{0,0,0,1}=-\frac{4U^{(k+1)}(0)^2}{k!(k-1)!}\iint \frac{1}{t^2-1} \left(t\arctanh\left(\frac{1}{t}\right)+\frac{1}{2}\ln(t^2-1)\right) dtdt-\\
\frac{2U^{(2k)}(0)}{(2k-1)!}\left(t\arctanh\left(\frac{1}{t}\right)+\frac{1}{2}\ln\left( t^2-1 \right)\right)=\\
\frac{2U^{(k+1)}(0)^2}{k!(k-1)!}\left( (t+1)(\ln(t-1)+1)\ln(t+1)-((2\ln 2+1)t-1)\ln(t-1)+\right.\\
\left.2t\hbox{dilog}((t+1)/2)\right)-
\frac{2U^{(2k)}(0)}{(2k-1)!}\left(t\arctanh\left(\frac{1}{t}\right)+\frac{1}{2}\ln\left( t^2-1 \right)\right)
\end{align*}
All the terms are in $\mathbb{C}[t,\arctanh\left( \frac{1}{t} \right),\ln\left( t^2-1 \right) ]$ except one, the dilogarithmic term
$$\dilog\left(\frac{t+1}{2}\right)=\int \frac{\ln(t+1)-\ln 2}{1-t} dt$$
The dilogarithm has a non commutative monodromy (see \cite{17}). As expected, the term in $U^{(2k)}(0)$ has a commutative monodromy. So the integrability constraint is that the dilogarithmic term should not appear. Then a necessary integrability constraint is that $U^{(k+1)}(0)=0$, which completes the recurrence. The function $U$ is meromorphic, all derivatives of $U$ are zero
$$U^{(k)}(0)=0 \;\forall k\in\mathbb{N}^*,$$
and so this implies that $U$ is constant. Then $V =r^{-1}$.
\end{proof}

To conclude, we have found all meromorphically integrable meromorphic homogeneous potentials on $\mathcal{C}$ which have a Darboux point with eigenvalues $-1$ (Proposition \ref{thm9}), $0$ (Lemma \ref{thm5}) and $2$ (Lemma \ref{thm7}). This implies Theorem \ref{thmmain2}.

\section{The other eigenvalues: $5,9,14,20,\dots$}\label{secThe}
To find all integrable potentials, we would need to study all the other possible eigenvalues. But for these larger eigenvalues, no integrable homogeneous potential of degree $-1$ is known. Then we can assume that such potentials do not exists, and we can also make a stronger assumption that at some order $k$, the $k$-th variational equation never has a Galois group whose identity component is Abelian (for any choice of derivatives of $V$ at $c$ of order $\geq 3$). 

\subsection{Axi-symmetric potentials}

We give here another application of the non degeneracy property. In the case of axi-symmetric potentials, the non degeneracy property has only to be checked for odd orders to imply a uniqueness result (because the odd order derivatives of $V$ are automatically zero). We obtain in particular the following result for axi-symmetric potentials.

\begin{thm}\label{thm11}
Let $V$ be a holomorphic homogeneous potential on $\Omega$ of degree $-1$. Assume that $V$ is invariant by the symmetry $q_2 \mapsto -q_2$, and that there exists a point of the form $c=(1,0,\dots)$ in $\Omega$. Then, up to dilatation, the set of such meromorphically integrable potentials is at most countable.
\end{thm}

\begin{rem}
Let first remark that the hypothesis of symmetry implies that the point $c=(1,0,\dots)\in \Omega$ is a Darboux point. We will prove in fact that for each eigenvalue (allowed by the Morales-Ramis Theorem) at this Darboux point, there is up to dilatation at most one such meromorphically integrable potential. We say nothing about their existence, and we only know them for $\lambda=-1,0,2$ which are respectively in polar coordinates
$$ W_0=\frac{1}{r}\qquad W_1=\frac{1}{r} \frac{1}{\cos(\theta)} \qquad W_2=\frac{1}{r} \frac{\cos(\theta)}{\cos(2\theta)}$$
\end{rem}

\begin{proof}
We just need to prove the non degeneracy property for odd orders. Indeed, for even orders, we use Euler relation which gives us all derivatives except one, the derivative in the normal direction to the straight line $\theta=0$ (see remark \ref{remhom}). But for the variational equation of even order $k$, this maximal order derivative is then of odd order $k+1$. This derivative is then automatically $0$ because we assume the symmetry. The non degeneracy is written
$$\frac{\partial}{\partial \alpha}  \underset{{t=\infty}}{\hbox{Res}}\;\; (t^2-1)^k (Q_n+\alpha\epsilon_n P_n)^{k+1} \neq0$$
We look at coefficient $\alpha^k$ of the above residue, i.e :
$$\epsilon_n^k(k+1)  \underset{{t=\infty}}{\hbox{Res}}\;\; (t^2-1)^k Q_n P_n^k=\frac{1}{2} \epsilon_n^k (k+1) \int\limits_{-1}^1 P_n^{k+1} dt$$
by using the Taylor expansion of $\arctanh\left( \frac{1}{t} \right)$ at infinity and recognizing that this sum can be written as this integral. The integer $k$ is odd, the polynomials $P_n$ are never identically $0$, then this coefficient never vanishes. This proves non degeneracy, and thus uniqueness  thanks to Lemma \ref{thm4}.
\end{proof}

\subsection{Computer experiments}

For a fixed eigenvalue, it is possible in practice to compute higher variational equations, then solve them (which comes down to search rational solutions, which can be done thanks to \cite{32,33}) and write explicitly the constraint on the highest order derivative of $V$. Below we made such a computation for $\lambda=5,9,14,20$ for variational equations up to of order $5$ or $7$.\\

\noindent
\textbf{Results}\\
For $\lambda=5,14$, we find that the only potentials integrable at order $4$ have the following series expansion
$$V_3=\frac{1}{r}\left(1+3\theta^2+\frac{125}{12}\theta^4+o(\theta^5)\right)$$ 
$$V_5=\frac{1}{r}\left(1+\frac{15}{2}\theta^2+\frac{374495}{5352} \theta^4+o(\theta^5)\right)$$
At order $5$, no possible solution is found: indeed, in section \ref{secHig}, we found that the integrability condition is of the form of affine equations in the $6$-th order derivative of $V$ (here $U^{(6)}(0)$). But at order $5$, two affine conditions are found, and they are incompatible. This proves in particular that a potential with eigenvalue $5,14$ is never integrable.\\

In the case $\lambda=9,20$, the condition of being integrable at order $2$ gives the following
$$V_4=\frac{1}{r}\left(1+5\theta^2+b\theta^3+o(\theta^3)\right)\qquad V_6=\frac{1}{r}\left(1+\frac{21}{2}\theta^2+b\theta^3+o(\theta^3)\right)$$
with an arbitrary $b$. In this case, the $2$-th order variational equation gives no constraint at all (this was already proved in \cite{18}), and thus the third order derivative of $U$ can be arbitrary. We can still continue to compute integrability constraints at higher orders, which are not trivial. At order $6$, we find the following possible series expansions
$$\frac{1}{5!}\partial_{\theta}^5V_4=\frac{363467}{4824000}b^3+\frac{112035}{8576}b$$
$$\frac{1}{6!}\partial_{\theta}^6V_4=\frac{216926052083}{10224685080000}b^4+\frac{279352141289}{54531653760}b^2+\frac{4715685295}{24563808}$$
$$\frac{1}{7!}\partial_{\theta}^7V_4=\frac{57826741017348283}{893377392990720}b+\frac{25932696791821703}{100504956711456000}b^3$$
with
$$R_4(b)=\frac{158469311}{97702546320000}b^4+\frac{372429603}{868467078400}b^2+\frac{45927}{2729312}=0$$

$$\frac{1}{5!}\partial_{\theta}^5V_6=\frac{68250852673}{4257725150000}b^3+\frac{98831601}{3475694}b$$
$$\frac{1}{6!}\partial_{\theta}^6V_6=\frac{10915637473609903}{5190230823727250000}b^4+\frac{6605928379884787}{1271076936423000}b^2+\frac{19638863047783}{10039110960}$$
\begin{align*}
\frac{1}{7!}\partial_{\theta}^7V_6=\frac{118828154548524498748866853503827777}{431797756299715943933989778480280}b+\\
\frac{8633140425176867273801758981735627411}{52895225146715203131913747863834300000}b^3
\end{align*}
with
\begin{align*}
 R_6(b)=\frac{198715111646995383}{2772435940648535262500000}b^4+\\
\frac{1448561702310687}{7921245544710100750}b^2+\frac{270431334600}{3128145145507}=0
\end{align*}

Remark the algebraic constraint on $b$: it appears at order $5$. Indeed, at order $5$, the integrability constraint is affine in the $6$-th order derivative of $U$, but is a polynomial of degree $4$ in $b$. The situation is similar to the eigenvalues $5,14$: we obtain two affine constraint in $U^{(6)}(0)$ which are compatible only if we carefully choose the parameter $b$ (as a root of $R_4,R_6$). At order $7$, we obtain again two affine conditions on $U^{(8)}(0)$, and this time we have no longer a free parameter: the potential can never be integrable at order $7$.\\

\noindent
\textit{Remark}. - These expansions of the potentials are unique and allow for each given potential to precisely compute the order at which it is integrable (here in the case of the eigenvalues $5,9,14,20$). Using this result, we can thus conclude that meromorphically integrable homogeneous potentials of degree $-1$ in the plane with eigenvalues $5,9,14,20$ do not exist. This computation strongly suggest that the same pattern will follow for higher eigenvalues, and thus that Theorem \ref{thmmain2} is in fact the complete list of integrable potentials. By this , we thus infer the two following conjecture
 
\begin{conj}
Let $V$ be a holomorphic homogeneous potential on $\Omega$ of degree $-1$ such that there exists a non-degenerate Darboux point $c\in\Omega$ with multiplier $-1$. If
$$\lambda(c)=\frac{1}{2}(n-1)(n+2) \quad n\hbox{ odd }n\geq 3$$
then $V$ is not integrable at order $5$ at $c$. If
$$\lambda(c)=\frac{1}{2}(n-1)(n+2) \quad n\hbox{ even }n\geq 4$$
then $V$ is not integrable at order $7$ at $c$.
\end{conj}

Solving this conjecture would classify completely integrable homogeneous potentials of degree $-1$ in the plane, and would probably allow with some generalization for other degrees to close completely the search of integrable homogeneous potentials (with at least some assumption on Darboux points). A partial proof up to the $5$-th variational equation would lead to classification of axi-symmetric integrable potentials and so it would imply for example that in Theorem \ref{thm11}, there are only $3$ axi-symmetric meromorphically integrable potentials, and thus that all of them are known. This would also lead to numerous theorems in higher dimension for potentials having discrete symmetry groups.

\section{Degenerate Darboux points}\label{secOth}
We have here only written about non degenerate Darboux points. Let us now look at degenerate Darboux points. The main result of this section is that they are useless for meromorphic integrability, and that even if we are looking only for rational integrability, their usefulness is limited: the only integrability condition coming from of Morales-Ramis-Simo Theorem is that they have to be multiple (which is equivalent to $\Sp(\nabla^2V(c))=\{0\}$). 

\begin{thm}\label{thm12}
Let $V$ be a rational homogeneous potential on $\Omega$ of degree $-1$. Assume there exists a degenerate Darboux point of the form $c=(1,0,\dots)\in\Omega$. If $V$ has a first integral $I$, rational on $\mathbb{C}^2\times\Omega$ and independent almost everywhere with $H$, then $\Sp(\nabla^2V(c))=\{0\}$. Conversely, if $\Sp(\nabla^2V(c))=\{0\}$, then the identity component of the Galois group of variational equation near the corresponding homothetic orbit is Abelian at any order.
\end{thm}

Here we add the restriction that $V$ is rational on $\Omega$. This is due to the fact that the variational equation (see the proof below) is not Fuchsian, and thus the Galois group over meromorphic functions could be different from the Galois group over rational functions.

\begin{proof}
We first write the potential in polar coordinates, on an open neighbourhood of $c$, $V(q_1,q_2,r,\mathbf{w})=r^{-1}U(\theta)$. The first order variational equation is the following
\begin{equation}\label{eqvardegenerate}
\ddot{X_1}=0\qquad \ddot{X_2}=\frac{U''(0)}{t^3}X_2
\end{equation}
Let $M$ be an open set of $\mathbb{C}^2\times \Omega$, containing the orbit
$$\Gamma= \{ q=(t c_1,t c_2),\;r=t,\; p=(c_1,c_2),\; t\in \mathbb{C}^* \} $$
and such that the Hamiltonian is holomorphic on $M$. To this orbit $\Gamma$, we add singular points $t=0,\infty$, noting it $\bar{\Gamma}$. We now use Theorem 2. of Morales-Ramis-Simo \cite{19} in its version with $\bar{\Gamma}$. As said in their article, this Theorem is still valid when adding singular points to the orbit $\Gamma$, and then considering the differential Galois group over the meromorphic functions on $\bar{\Gamma}$ (see also Morales-Ramis \cite{5} p 114). If the potential $V$ is rationally integrable (and thus meromorphic on a neighbourhood of $\bar{\Gamma}$, the variational equation \eqref{eqvardegenerate} should have a Galois group with an Abelian identity component over the base field of meromorphic functions on $\bar{\Gamma}$. As $\bar{\Gamma}\simeq \bar{\mathbb{C}}$, this base field is the rational functions $\mathbb{C}(t)$.

Let us now compute the Galois group of equation \eqref{eqvardegenerate}. The first equation is clearly integrable. Assume now that $U''(0)\neq 0$. For the second one, we make a linear variable change and this gives
\begin{equation}\label{eq10}
\ddot{y}=\frac{1}{t^3}y
\end{equation}
Using the Kovacic algorithm, we find that the Galois group of this equation is $SL_2(\mathbb{C})$, and thus connected and non Abelian. So the only possibility left is $U''(0)= 0$, for which equation \eqref{eqvardegenerate} has a Galois group equal to $\{Id\}$ (thus Abelian).

Now assume the reverse, that $U''(0)=0$. The first order variational equation is written
$$\ddot{X_1}=0\qquad \ddot{X_2}=0$$
We already know that Morales-Ramis-Simo integrability condition is satisfied at order $1$, and we know want to test it at any order.

\begin{lem}
The algebra $\mathcal{A}=\mathbb{C}[t,\frac{1}{t},\ln t]$ is stable by integration.
\end{lem}

\begin{proof} We consider $f\in \mathbb{C}[t,\frac{1}{t},\ln t]$ and we write it as a linear combination of terms of the type
$$t^n \ln(t)^m\quad n\in\mathbb{Z},\; m\in\mathbb{N}$$
If $n\geq 0$, then we use integration by parts to decrease $m$ until $0$. If $n<0$, We use integration by part to increase $n$ up to $n=-1$. We then have the formula
$$\int \frac{1}{t} \ln(t)^m dt=\frac{1}{m+1} \ln(t)^{m+1}$$
Then all functions in $\mathbb{C}[t,\frac{1}{t},\ln t]$ have a primitive in $\mathbb{C}[t,\frac{1}{t},\ln t]$.
\end{proof}

We now use this Lemma, remarking the following phenomenon. The solutions of higher variational equation are in fact solutions of non homogeneous linear differential equations and the non homogeneous terms are produced only using products of lower order solutions and functions $t^{-k}$. So the solutions always live in some algebra in which we take recursively integrations. So we apply the method of variation of constants to find the solutions. Moreover, the Wronskian of $\ddot{X_2}=0$ is equal to $1$ (and also for the higher variational equations matrices), then we never have to divide. So all solutions live in the algebra $\mathcal{A}$ which is stable by integration. Then the Picard Vessiot field is
$$K_i=\mathbb{C}(t) \hbox{ or } \mathbb{C}(t,\ln\; t)\quad G_i= \{id\} \hbox{ or }\mathbb{C}$$
The Galois group is in both cases Abelian at any order.
\end{proof}

\begin{rem}
We were only able to analyse rational first integrals. Here the variational equation \eqref{eqvardegenerate} is not Fuchsian at $0$. This condition is a hypothesis of our Lemma 3 in \cite{40} which proves that the Galois groups over the meromorphic functions and the Galois group over the rational functions are equal. Thus, we cannot use this Lemma, and we need to use Morales-Ramis-Simo Theorem over $\bar{\Gamma}$. So only the rational first integrals can be analysed. To have a ``reasonable definition of integrability'', we than need to assume that $V$ is rational, as if the potential $V$ itself is not meromorphic for $r=0$, for example when $U$ is not a rational function in $\exp i\theta$, then the Hamiltonian (the only first integral we know in advance) would not be rational and then excluded from this analysis.
\end{rem}

This type of proof appears to be very general. Indeed, if some potential appears to be integrable at all order near a particular solution, without known first integral, the Picard Vessiot field is often not growing. If the coefficients of the potential are not well adjusted to avoid creating further monodromy, it is probably because it is not possible. Looking at the non linear version of variational equation in \cite{19} p 860 (and also here section \ref{secHig}), we see that the solutions of higher variational equations are in fact solutions of non homogeneous linear differential equations and the non homogeneous terms are produced only using products. As such equation can be solved using the method of variation of parameters, the solutions always live in some algebra in which we take recursively integrations. For homogeneous potentials of degree $-1$ and non degenerate Darboux points, the algebra is given by the following process
$$\mathcal{A}_0=\mathbb{C}\left[t,\frac{1}{t^2-1}\right],\qquad \mathcal{A}_{i+1}=\int \mathcal{A}_i dt$$
where $\int \mathcal{A}_i dt \supset \mathcal{A}_i $ is the algebra generated by all integrations of functions in $A_i$. We have then in particular that the Picard Vessiot of $(VE_i)$ (in the case of all the eigenvalues belong to Morales Ramis table) is always contained in the fraction field of $A_{i+1}$. These algebras contain in particular all the polylogarithms functions that give integrability constraints, but not only them.

\section{Conclusion}
We completely analysed integrability for $3$ infinite dimensional families of potentials, corresponding to eigenvalues $-1,0,2$. Through numerical computations, we conjecture that there are no integrable homogeneous potentials of degree $-1$ in the plane with other eigenvalues, and thus that the list of integrable potentials of Theorem \ref{thmmain2} is complete. This conjecture represents the last open question about homogeneous potentials in the plane of degree $-1$, outside of existence of Darboux points. It can be tested for finitely many eigenvalues, but testing them for all possible eigenvalues at once seems difficult. In principle, they could be checked using the $D$-finiteness property of the functions $P_n,Q_n$ (the fact that they satisfy linear polynomial recurrence and differential equations), but in practice direct computation seems to be way out of reach for the moment.\\

\noindent
\textbf{Some examples of potentials integrable at all order near all Darboux points}\\
In the case of non degenerate Darboux points, if we admit conjecture $2$, it will not be possible to find non integrable potentials which are integrable all orders. We then need to find homogeneous potentials either having no Darboux points at all, either only multiple degenerate Darboux points (the second derivative should vanish). The functions $U(\theta)=F\left(e^{i\theta}\right)$ with
$$F(z)=h(z^n)\quad h\hbox{ Moebius transformation, } n\in\mathbb{N}^*$$
$$F(z)=f(z^n) \hbox{ with }f(z)=\int \frac{a z^i}{(z-\alpha)^j} dz \;\;0\leq i\leq j-2,\;n\in\mathbb{N}^* \; \alpha\in\mathbb{C}^*$$
have no critical points. The functions $U(\theta)=F\left(e^{i\theta}\right)$ with
$$F(z)=h((z^n-\alpha)^m)-h(0)\quad m\geq3,\;n\in\mathbb{N}^*\;\alpha\in\mathbb{C}^*$$
with $h$ a Moebius transformation, have only degenerate Darboux points satisfying the integrability constraint.

\bigskip

These examples show that there are still open questions about integrability, but the difficulties do not rely on Morales Ramis theory but on the search of Darboux points. Potentials without non degenerate Darboux points are not common, but they still exist and a complete classification of them seems to be difficult.

\nocite{*}
\bibliographystyle{plain}
\bibliography{classificationpotentiels}
\end{document}